\documentclass{bcp}
\usepackage{amsmath, amssymb, amscd, amsthm, amsbsy, euscript} 
\usepackage{graphicx} 

\newtheorem{theor}{Theorem}
\newtheorem{state}[theor]{Proposition}
\newtheorem{cor}[theor]{Corollary}
\newtheorem{lemma}[theor]{Lemma}
\theoremstyle{definition}
\newtheorem{define}{Definition}
\newtheorem{example}{
{Example}}
\newtheorem{Rule}{Rule}

\newtheorem*{convent}{Convention}
\theoremstyle{remark}
\newtheorem{rem}{\textmd{Remark}}

\def\oldvec{\mathaccent "017E\relax }
\DeclareMathOperator{\Or}{\mathsf{O\oldvec{r}}}

\newcommand{\cev}[1]{\reflectbox{\ensuremath{\vec{\reflectbox{\ensuremath{#1}}}}}}

\newcommand{\BBR}{\mathbb{R}}
\newcommand{\BBN}{\mathbb{N}}

\newcommand{\BBZ}{\mathbb{Z}}

\newcommand{\cP}{\mathcal{P}}\newcommand{\cQ}{\mathcal{Q}}

\newcommand{\cX}{{\EuScript X}}    
\newcommand{\cZ}{{\EuScript Z}}    

\newcommand{\bx}{{\boldsymbol{x}}}

\newcommand{\bxi}{{\boldsymbol{\xi}}}

\newcommand{\veps}{\varepsilon}

\newcommand{\dd}{\partial}
\newcommand{\Id}{{\mathrm d}}

\newcommand{\sE}{\mathsf{E}}
\newcommand{\vbDelta}{\vec{\boldsymbol{\Delta}}}
\DeclareMathOperator{\Vertices}{Vert}
\DeclareMathOperator{\Edge}{Edge}

\DeclareMathOperator{\Sym}{Sym}
\DeclareMathOperator{\Aut}{Aut}

\DeclareMathOperator{\sign}{sign}
\DeclareMathOperator{\Gra}{Gra}

\DeclareMathOperator{\End}{End}

\newcommand{\lshad}{[\![}
\newcommand{\rshad}{]\!]}

\newcommand{\by}[1]{\textrm{{#1}}}
\newcommand{\jour}[1]{\textrm{{#1}}}
\newcommand{\vol}[1]{\textrm{{#1}}}
\newcommand{\book}[1]{\textit{{#1}}}

\begin{document}


\mathclass{Primary 
05C22, 
16E45, 
53D17; 
Secondary 
68R10, 
81R60.
}

\abbrevauthors{A. V. Kiselev and R. Buring}

\abbrevtitle{The graph orientation morphism revisited}

\title{The Kontsevich graph orientation morphism revisited}

\author{Arthemy V. Kiselev}
\address{Bernoulli Institute for Mathematics,
Computer Science and Artificial Intelligence,\\ 
University of Groningen,
P.O.Box~407, 9700\,AK Groningen, 
The Netherlands.\\
E-mail: A.V.Kiselev@rug.nl (corresponding author)}

\author{Ricardo Buring}
\address{Institut f\"ur Mathematik, 
Johannes Gutenberg\/--\/Uni\-ver\-si\-t\"at,\\
Staudingerweg~9, 
\mbox{D-\/55128} Mainz, Germany.\\
E-mail: rburing@uni-mainz.de}

\maketitlebcp

\begin{abstract}
The orientation morphism $\Or(\cdot)(\cP)\colon\gamma\mapsto\dot{\cP}$ associates differential\/-\/polynomial flows $\dot{\cP}=\cQ(\cP)$ on spaces of bi\/-\/vectors~$\cP$ on finite\/-\/dimensional affine manifolds~$N^d$ with (sums of) finite unoriented graphs~$\gamma$ with ordered sets of edges and without multiple edges and one\/-\/cycles. It is known that $\Id$-\/cocycles~$\boldsymbol{\gamma}\in\ker\Id$ with respect to the vertex\/-\/expanding differential~$\Id=[{\bullet}\!\!{-}\!{-}\!\!{\bullet},\cdot]$ are mapped by~$\Or$ to Poisson cocycles $\cQ(\cP)\in\ker\,\lshad\cP,{\cdot}\rshad$, that is, to in\-fi\-ni\-te\-si\-mal symmetries of Poisson bi\/-\/vectors~$\cP$.
The formula of orientation morphism~$\Or$ was expressed 
in terms of the edge orderings as well as parity\/-\/odd and parity\/-\/even derivations on the odd cotangent bundle $\Pi T^* N^d$ over any $d$-\/dimensional affine real Poisson manifold~$N^d$.
We express this formula
in terms of 
(un)oriented graphs themselves, \textsl{i.e.}\ without explicit reference to supermathematics on~$\Pi T^* N^d$.
\end{abstract}

\subsection*{Introduction}
A differential graded Lie algebra structure on the vector space of unoriented graphs with ordered sets of parity\/-\/odd edges was introduced by Kontsevich in~\cite{MKParisECM,MKZurichICM}. The vertex\/-\/expanding differential $\Id=[\bullet\!{-}\!\bullet,{\cdot}]$ is the adjoint action by the edge~$\bullet\!{-}\!\bullet$, which itself satisfies the master equation $[\bullet\!{-}\!\bullet,\bullet\!{-}\!\bullet]=(\text{const}\neq0)\cdot\mathbf{0}\in$ Graph complex. Examples of $\Id$-\/cocycles are discussed in~\cite{Ascona96,WillwacherZivkovic2015Table} or~\cite{JNMP2017}. Properties of the unoriented graph complex and its relation to the Grothendieck\/--\/Teichm\"uller group were explored by Willwacher 
(see~\cite{WillwacherGRT}).

The language of 
graphs allows 
encoding (poly)differential operators on affine manifolds (let us consider~$\mathbb{R}^d$ for simplicity), as well as encoding graded\/-\/symmetric endomorphisms $\pi\in\End 
\bigl(T_{\text{poly}}(\BBR^d)\bigr)$ on the spaces of totally skew\/-\/symmetric 
multivector 
fields on~$\BBR^d$. 
For example, a Poisson bivector~$\cP$ on~$\BBR^d$ is represented by the wedge~$\bigwedge$
with Left~$<$ Right edge ordering, 
whereas the Schouten bracket is an endomorphism:
\begin{multline*}
\pi_S(F,G) =  \Or\bigl( \stackrel{\mathsf{1}}{ {\bullet} } \!\!\!{-}\!{-}\!\!\! \stackrel{\mathsf{2}}{ {\bullet} } 
\bigr)(F\otimes G)
= \tfrac{1}{2!} \Bigl( \frac{\vec{\dd}}{\dd\bxi}{\Bigr|}_{\mathsf{1}}\otimes
\frac{\vec{\dd}}{\dd\bx}{\Bigr|}_{\mathsf{2}} +
\frac{\vec{\dd}}{\dd\bxi}{\Bigr|}_{\mathsf{2}}\otimes
\frac{\vec{\dd}}{\dd\bx}{\Bigr|}_{\mathsf{1}} \Bigr)
\bigl( F{\bigr|}_{\mathsf{1}}\cdot G{\bigr|}_{\mathsf{2}} + 
F{\bigr|}_{\mathsf{2}}\cdot G{\bigr|}_{\mathsf{1}} \bigr) ={}\\
{}=
\frac{\vec{\dd}}{\dd\bxi}(F)\cdot \frac{\vec{\dd}}{\dd\bx}(G)
+(-)^{|F|} \frac{\vec{\dd}}{\dd\bx}(F)\cdot \frac{\vec{\dd}}{\dd\bxi}(G) 
=
 (-)^{|F|-1}\Bigl\{ (F)\frac{\cev{\dd}}{\dd\bxi} \cdot\frac{\vec{\dd}}{\dd\bx}(G)
 - (F)\frac{\cev{\dd}}{\dd\bx} \cdot\frac{\vec{\dd}}{\dd\bxi}(G) \Bigr\} ={}\\
{}=
(-)^{|F|-1}\lshad F,G\rshad = (-)^{|F|-1} \bigl\{ F \xrightarrow{\phantom{mmm}} G
\ \ -\ \ F \xleftarrow{\phantom{mmm}} G \bigr\},
\end{multline*}
for any homogeneous multivectors $F$ and~$G$.
It is readily seen from the definition that the endomorphism $\pi_S$ is graded\/-\/symmetric, \[\pi_S(F,G)=(-)^{|F|\cdot|G|}\pi_S(G,F),\] and the usual variant of the Schouten bracket is shifted\/-\/graded skew\/-\/symmetric, 
\[\lshad F,G\rshad=-(-)^{(|F|-1)\cdot(|G|-1)}\lshad G,F\rshad.\]
This language of oriented graphs is used in Kontsevich's solution of the problem of deformation quantisation on finite\/-\/dimensional Poisson manifolds (see~\cite{Ascona96,KontsevichFormality} and~\cite{sqs15,cpp,Kiev18}). It provides also the construction of universal --\,with respect to all affine Poisson manifolds\,-- infinitesimal symmetries $\dot{\cP}=\cQ(\cP)$ of classical Poisson structures (see~\cite{Ascona96,Bourbaki2017} and~\cite{f16}).

The orientation morphism~$\Or$ associates Poisson cocycles $\cQ(\cP)\in\ker\lshad\cP,\cdot\rshad$ with $[\bullet\!{-}\!\bullet,\cdot]$-\/cocycles in the unoriented graph complex. The algebraic formula of~$\Or$ was given in~\cite{Ascona96}; its constructions was discussed in more detail by Jost in~\cite{Jost2013}, by Willwacher in~\cite{WillwacherGRT}, and by the authors in a recent paper~\cite{OrMorphism} (cf.~\cite{sqs19}).
This morphism is a tool which produces symmetries of Poisson brackets. 
If such a symmetry is not Poisson\/-\/exact in the second Poisson cohomology, then it yields deformations of classical Liouville\/-\/integrable systems, preserving the property that the dynamics is Poisson. 
If the symmetry of a given bracket is Poisson\/-\/exact, then its flow produces a family of diffeomorphisms --\,typically, nonlinear\,-- of the affine Poisson manifold under study.
A construction of universal Poisson $1$-\/cocycles $\vec{\cX}=\Or(\gamma)(\vec{V},\cP^{\otimes^{n-1}})$ from graph cocycles $\gamma=\sum_a c_a\cdot\gamma_a$ with $n$ vertices in each term~$\gamma_a$ is introduced in~\cite{sqs19} for homogeneous Poisson bi\/-\/vectors $\cP=\mathrm{L}_{\vec{V}}(\cP)=\lshad\vec{V},\cP\rshad$.
(The domain of definition of~$\Or$ is of course larger than the homogeneous component of unoriented graphs on $n$ vertices and $2n-2$ edges, $n\in\BBN$.) So far, one example of its work, namely $\Or\colon \boldsymbol{\gamma}_3\in\ker\Id\mapsto \cQ_{1:\frac{6}{2}}(\cP)\in\ker\partial_\cP$, was known from the seminal paper~\cite{Ascona96} in which the tetrahedral flow was introduced (cf.~\cite{Bourbaki2017} and~\cite{f16}). The pentagon\/-\/wheel flow $\Or(\boldsymbol{\gamma}_5)(\cP)$ has been obtained in~\cite{sqs17} by solving the factorisation problem $\lshad\cP,\cQ_5(\cP)\rshad=\Diamond\bigl(\cP,\lshad\cP,\cP\rshad\bigr)$ with respect to the Leibniz graphs in~$\Diamond$ and the Kontsevich bivector graphs in $\cQ_5=\Or(\boldsymbol{\gamma}_5)$ on 6~internal vertices.
The heptagon\/-\/wheel flow $\cQ_7=\Or(\boldsymbol{\gamma}_7)$ on 8~internal vertices is described in~\cite{OrMorphism}.

The algebraic formula of the orientation morphism amounts to a simple but extensive calculation using $\BBZ/2\BBZ$-gradings, \textsl{i.e.}\ supermathematics. The morphism $\Or$ determines both the orgraph multiplicities and the signs of all the Kontsevich graphs $\Gamma$ (which, we recall, are equipped with an ordering Left~$<$ Right at all their internal vertices~$v$, such that $\Gamma\bigl(L<R{\bigr|}_v\bigr)=-\Gamma\bigl(R<L{\bigr|}_v\bigr)  
$). Nevertheless, the algebraic formula is external with respect to the (un)oriented graph complexes. 
We pose the question 
whether doing this 
$\BBZ_2$-\/graded 
computation --\,involving thousands and millions of graphs\,-- 
is the only way to  
build each element of a known infinite sequence $\dot{\cP}=\cQ_{2\ell+1}(\cP)$ of 
flows.\footnote{Willwacher showed in~\cite{WillwacherGRT} that at every $\ell\in\BBN$, the $(2\ell+1)$-\/wheel graph marks a nontrivial $\Id$-\/cocycle in the unoriented graph complex, and those are expected to yield --\,for generic Poisson structures~$\cP$\,--
nontrivial terms in the respective Poisson cohomology under the orientation morphism. 
It remains an open problem to find a graph cocycle~$\gamma$ and Poisson bracket~$\cP$ on an affine manifold~$N^d$ such that the respective Poisson cocycle $\Or(\gamma)(\cP)$ is nontrivial.
Iterated commutators of the cocycles~$\boldsymbol{\gamma}_{2\ell+1}$ yield 
Poisson cocycles as well.}

The aim of this note is to express the orientation morphism in geometric terms, that is, by using combinatorial data which are intrinsic with respect to the (un)oriented graphs~$\gamma_a$ equipped with edge orderings $I\prec II\prec\ldots$ in~$\gamma_a$ and wedge orderings $L_i<R_i$ at vertices~$v_i$ in~$\Or(\gamma_a)$, respectively.
We let a cocycle $\gamma=\sum_a c_a\cdot\gamma_a\in\ker\Id$ on $n$~vertices and with a given 
ordering $I\prec II\prec III\prec\ldots$ of $2n-2$ edges in each graph~$\gamma_a$
be the initial datum.
The problem is to reveal the few rules of matching for the Kontsevich oriented graphs: whenever a sign 
is fixed in front of just one Kontsevich graph which is obtained by 
orienting a given graph~$\gamma_a$ in a cocycle $\gamma=c_1\gamma_1+\ldots+c_m\gamma_m\in\ker\Id$, the signs in front of \textsl{all} admissible orientations of \textsl{all} the graphs in their linear combination~$\gamma$ are then 
determined without calculation of~$\Or$ for each of them. 

This paper contains three sections.
For consistency, in section~\ref{SecAlgFormula} we recall the algebraic formula of the orientation morphism~$\Or$ for the Kontsevich graph complex; the approach is operadic 
(cf.~\cite{Operads1999}).
In section~\ref{SecMultSignRules}, we establish the main rule for the count of multiplicities and signs of orgraphs.
In the final section, we expand on the rule of signs purely in terms of oriented Kontsevich graphs; to this end, we analyse their admissible types and combinatorial properties of transitions between them.

\section{The algebraic formula of graph orientation morphism~$\Or$}\label{SecAlgFormula}
\noindent%
Denote by $\Gra$ the real vector space of (formal sums\footnote{For the sake of definition, we assume by default that all terms in each sum are homogeneous w.r.t.\ the various gradings.}
of) unoriented, --\,by default, connected\,-- finite graphs~$\gamma=\sum_a c_a\cdot\gamma_a$ equipped with wedge ordering of edges $\sE(\gamma_a)=I\wedge II\wedge\ldots$ in each graph~$\gamma_a$; by definition, put $(\gamma$, $I\wedge II\wedge\ldots)=-(\gamma$, $II\wedge I\wedge\ldots)$, etc.
The vertices of a graph~$\gamma_a$ are not ordered; one is free to choose any labelling of 
these $\#\Vertices(\gamma_a)$ vertices by using $1$, $\ldots$, $n=\#\Vertices(\gamma_a)$ because the assignment $\wp_i\mapsto v_{\sigma(i)}$ of an $n$-\/tuple $\wp_1$, $\ldots$, $\wp_n$ of multivectors, which form 
the ordered set of graded arguments for an endomorphism defined in formula~\eqref{EqOr} below, will be provided by the entire permutation group $S_n\ni\sigma$: the $i$th multivector is placed into the vertex~$v_{\sigma(i)}$.

A \textsl{zero unoriented graph}~$\boldsymbol{0}$ is a graph which is equal to minus itself under an automorphism (that induces a permutation of edges whose parity we thus inspect). 

\begin{example}
The two\/-\/edge connected graph $\bullet\!{-}\!{\bullet}\!{-}\!\bullet$ is a zero graph because the flip, a symmetry of this graph, swaps the edges which anticommute, $I\wedge II=-II\wedge I$, whence the graph at hand equals minus its image under its own automorphism, i.e.\ minus itself.
\end{example}

The operadic insertion ${\circ}_i\colon (\gamma_1$, $\sE(\gamma_1))\otimes(\gamma_2$, $\sE(\gamma_2))\mapsto(\gamma_1\mathbin{{\circ}_i}\gamma_2$, $\sE(\gamma_1)\wedge \sE(\gamma_2))$ is the Leibniz\/-\/rule sum of insertions of~$\gamma_1$ into vertices of~$\gamma_2$ such that 
every 
edge which was incident to a vertex~$v$ in~$\gamma_2$ is redirected to vertices in~$\gamma_1$ in all possible ways according to another Leibniz rule. The super Lie bracket $[\gamma_1,\gamma_2]=\gamma_1\mathbin{{\circ}_i}\gamma_2-(-)^{\#\sE(\gamma_1)\cdot\#\sE(\gamma_2)}\gamma_2\mathbin{{\circ}_i}\gamma_1$ on~$\Gra$ yields the vertex\/-\/expanding differential~$\Id=[\bullet\!{-}\!\bullet,\cdot]$.

Denote by $\End 
\bigl(T_{\text{poly}}(\BBR^d)\bigr)$ the vector space of 
graded\/-\/symmetric endomorphisms on the space of 
multivector 
fields on the space~$\BBR^d$ (which we view here as an affine manifold). In a full analogy with graphs (see~\cite{OrMorphism}), one inserts a given endomorphism into an argument slot of another endomorphism, proceeding over the slots consecutively and taking the sum.
The Richardson\/--\/Nijenhuis bracket $[\cdot,\cdot]_{\text{RN}}$ is the 
graded 
symmetrisation (with respect to its multivector arguments) of the difference of insertions: 
$[\pi,\rho]=\Sym 
\bigl(\pi\circ\rho-(-)^{|\pi|\cdot|\rho|}
\rho\circ\pi\bigr)$.
The Schouten bracket $\pi_S(R,S) = (-)^{|R|-1}\,
\lshad R 
,S 
\rshad 
$ of multivectors
~$R$,\ $S$ is a 
graded 
symmetric
endomorphism which satisfies the (shif\-ted\mbox{-)} graded Jacobi identity $[\pi_S,\pi_S]_{\text{RN}}=0$; hence by the graded Jacobi identity for the Richardson\/--\/Nijenhuis bracket (see~\cite{Identities18}), the operation~$[\pi_S,\cdot]_{\text{RN}}$ is a differential.

For a given graph $\gamma$ on $n$ vertices and for an ordered $n$-\/tuple $\boldsymbol{\wp}$ of multivectors $\wp_1$, $\ldots$, $\wp_n$,
the morphism~$\Or\colon\Gra\to\End 
\bigl(T_{\text{poly}}(\BBR^d)\bigr)$ is encoded by the formula
\begin{equation}\label{EqOr}
\Or(\gamma)(\wp_1\otimes\dots\otimes\wp_{ n 
}) =
\tfrac{1}{n!} \sum\limits_{\sigma\in S_n}
\prod\limits_{e_{ij}\in \sE(\gamma)} \vec{\Delta}_{ij}\, 
\bigl(\wp_1{\bigr|}_{v_{\sigma(1)}} \cdot\ldots\cdot \wp_n 
 {\bigr|}_{v_{\sigma(n)}} \bigr),
\end{equation}
where, in the ordered product of edges, the first edge $I\in \sE(\gamma)$ acts first and the last edge acts last. For every edge $e_{ij}$ connecting the vertices $v_i$ and~$v_j$ in the graph~$\gamma$, the edge operator, 
\[
\vec{\Delta}_{ij}=
\sum_{\alpha=1}^{\dim\BBR^d} \bigl\{ i \xrightarrow{\alpha} j + i \xleftarrow{\alpha} j \bigr\} = 
\sum_{\alpha=1}^{\dim\BBR^d} \Bigl(
\frac{\vec{\dd}}{\dd \xi_\alpha^{(i)}}\otimes \frac{\vec{\dd}}{\dd x^\alpha_{(j)}} 
 +
\frac{\vec{\dd}}{\dd x^\alpha_{(i)}}\otimes \frac{\vec{\dd}}{\dd \xi_\alpha^{(j)}} 
\Bigr)
\]
acts from 
left to right by the (graded-) derivations along the ordered sequence of multivectors
$\wp_1\otimes\ldots\otimes\wp_n$, now placed by a permutation $\sigma\in S_n$ into the respective vertices $v_{\sigma(1)}$, $\ldots$, $v_{\sigma(n)}$.
In this way, the derivations in each edge operator $\vec{\Delta}_{ij}$ reach the content
$\wp_{\sigma^{-1}(i)}$ and $\wp_{\sigma^{-1}(j)}$ of the vertices $v_i$ and~$v_j$.
Summarizing, the graph with an edge ordering $(\gamma,\sE(\gamma))$ is the set of topological data which determine all the portraits of 
couplings within the $n$-\/tuple of multivectors:
the above formula is referred to a $d$-\/tuple $\{x^\alpha\}$ of affine coordinates on~$\BBR^d$ and the canonical conjugate $d$-\/tuple $\{\xi_\alpha\}$ of fibre coordinates on the parity\/-\/odd cotangent bundle~$\Pi T^*\BBR^d$, so that the summation $\sum_{\alpha_{ij}=1}^{\dim\BBR^d}$ gives the coupling along the edge~$e_{ij}$.
(Clearly, one proceeds by linearity over sums of graphs~$\gamma$ and over sums of multivectors $\wp_i$ at every $i$ running from $1$ to~$n$.)

\begin{lemma}\label{LemmaZeroWellDefIf}
Suppose that at most one --\,in the rest of this paper, none usually, still exactly one 
only in Corollary~\ref{CorDiamond} below and likewise in~\cite{sqs19}\,-- multivector $\wp_1$, $\ldots$, $\wp_{\#\Vertices(\cZ\in\Gra)}$ is odd\/-\/graded, while all the rest are even\/-\/grading (so that all are permutable without signs appearing).
Then, for the sets of endomorphisms' arguments restricted in this way,
the morphism~$\Or$ in~\eqref{EqOr} is well defined on the space~$\Gra$, that is, its 
action gives 
\begin{multline*}
\Or\,(\text{zero graph }\cZ) ={}\\ 
{}= (0\in\BBR)\cdot(\text{nonzero oriented graphs})+(\text{coeffs}\in\BBR)\cdot(\text{zero Kontsevich orgraphs}).
\end{multline*}
\end{lemma}

\begin{proof}
By definition, a zero unoriented graph~$\cZ$ has a symmetry $\sigma\in\Aut(\cZ)$ which acts by a parity\/-\/odd permutation~$\sigma_E$ on the wedge\/-\/ordered set~$\sE(\cZ)$ of its edges and by a permutation~$\sigma_V$ on the set~$\Vertices(\cZ)$ of its vertices.
The permutation of edges yields the same parity\/-\/odd permutation~$\sigma_E$ of 
the edge operators~$\Delta_{ij}$ in their ordered product $\vbDelta = \prod_{e_{ij}\in\sE(\cZ)}\vec{\Delta}_{ij}$.
The permutation $\sigma_V$ of the content of vertices does not yield any signs because by our assumption, all the multivectors are pairwise permutable.
Put $n=\#\Vertices(\cZ)$.
Acting by the symmetry~$\sigma$ of unoriented graph~$\cZ$ on the anticommuting 
edge operators and on the 
multivectors $\wp_1$, $\ldots$, $\wp_n$, which are assigned consecutively by all the $n!$ permutations $\tau\in S_n$ to 
the vertices of~$\cZ$ (and the ordered product of which is the argument of~$\vbDelta$), we obtain --\,for the isomorphic unoriented graph~$\sigma(\cZ)$\,-- the expression
\begin{multline*}
\sum_{\tau\in S_n} \prod_{e_{ij}\in\sE(\cZ)} \sigma_E(\vec{\Delta}_{ij})\: 
\Bigl( \prod_{\Vertices(\cZ)} \tau\bigl(\sigma_V(\boldsymbol{\wp})\bigr) \Bigr) 
={}\\
{}=  
\sum_{\tau\in S_n} \prod_{e_{ij}\in\sE(\cZ)} \vec{\Delta}_{\sigma_E(e_{ij})} \: 
\Bigl( \prod_{\Vertices(\cZ)} \tau\bigl(\sigma_V(\boldsymbol{\wp})\bigr) \Bigr) 
=
\sum_{\tau\in S_n} \bigl( - \vbDelta)\: \bigl( +\!\!\prod_{ 
\Vertices(\cZ)} \tau(\boldsymbol{\wp})   
\bigr).
\end{multline*}
This confirms that $\Or(\cZ)(\wp_1$,\ $\ldots$,\ $\wp_n)$ is a differential polynomial (w.r.t.\ the components of multivectors~$\wp_i$) which equals minus itself, hence it is identically zero. But this 
defining property of the differential\/-\/polynomial expression
means that the Kontsevich orgraph that encodes the endomorphism $\Or(\cZ)$ is a zero orgraph.
\end{proof}

\begin{state}
Under the assumption --\,similar to the above in Lemma~\ref{LemmaZeroWellDefIf}\,-- that all the multivector arguments $\wp_i$ of the endomorphisms are in fact copies of a given Poisson bi\/-\/vector~$\cP$, 
the morphism~$\Or$ ``respects the operadic insertions'' 
by sending the super Lie bracket of graphs to the Richardson\/--\/Nijenhuis bracket of 
endomorphisms: 
\[
\Or\bigl([\gamma_1,\gamma_2]\bigr)=\bigl[\Or(\gamma_1),\Or(\gamma_2)\bigr]_{\text{RN}},
\]
where the right\/-\/hand side is graded 
symmetrised by construction.

In particular, $\Or(\bullet\!{-}\!\bullet)=\pi_S$,
i.e.\ the orientation morphism sends the edge to the Schouten bracket, whence
\begin{equation}\label{EqDiffToSchoutenWith}
\Or\bigl( 
\Id(\gamma)
\bigr) = \bigl[\pi_S,\Or(\gamma)\bigr]_{\text{RN}}.
\end{equation}
This is standard (see Appendix~\ref{AppLieAlgMor} below and also
~\cite{Jost2013,Ascona96,Operads1999} or~\cite{OrMorphism}).
\end{state} 

\begin{cor}\label{CorDiamond}
Evaluating the endomorphisms 
in both sides of Lie superalgebra morphism~\eqref{EqDiffToSchoutenWith}
at 
$n+1$ 
copies of a given Poisson bivector~$\cP$ satisfying $\lshad\cP,\cP\rshad=0$ on~$\BBR^d$,
\[
\Or\bigl([\bullet\!{-}\!\bullet,\gamma]\bigr)(\cP,\ldots,\cP) =
2 \lshad\cP,\Or(\gamma)(\cP,\ldots,\cP)\rshad -
\sum_i \Or(\gamma)\bigl(\cP,\ldots,\cP,\smash{\underbrace{\lshad\cP,\cP\rshad,}_{\text{$i$th slot}}}\cP,\ldots,\cP\bigr),
\]
one obtains 
--\,for a cocycle~$\gamma$ on $n$~vertices and $2n-2$~edges\,--
an explicit solution of the factorisation problem,
\begin{equation}\label{EqDiamond}
\lshad\cP,\Or(\gamma)(\cP)\rshad=\Diamond\bigl(\cP,\lshad\cP,\cP\rshad\bigr),
\end{equation}
for the Poisson cocycles $\cQ(\cP)=\Or(\gamma)(\cP)$ corresponding to $[\bullet\!{-}\!\bullet,\cdot]$-\/cocycles in the unoriented graph complex.

Otherwise speaking, the orientation morphism~$\Or$ sends $\Id$-\/co\-cy\-cle graphs to $\dd_\cP$-\/co\-cy\-c\-les (\textsl{i.e.}\ Poisson cocycles). Moreover, $\Id$-\/exact unoriented graphs $\gamma=\Id(\zeta)$ yield $\dd_\cP$-\/co\-boun\-da\-ries $\cQ(\cP)=\lshad\cP,\cX(\cP)\rshad$ with one\/-\/vectors $\cX(\cP)=2\cdot\Or(\zeta)(\cP)$
(modulo the improper terms which, by definition, vanish on the entire~$\BBR^d$ whenever a bivector~$\cP$ is Poisson).
\end{cor}

\begin{example}
The factorisation operators~$\Diamond$ in~\eqref{EqDiamond} for the tetrahedral flow $\dot{\cP}=\cQ_{1:\frac{6}{2}}(\cP)=\Or(\boldsymbol{\gamma}_3)(\cP)$, 
pentagon\/-\/wheel flow $\Or(\boldsymbol{\gamma}_5)(\cP)$,
and for the heptagon\/-\/wheel flow $\Or(\boldsymbol{\gamma}_7)(\cP)$
have been presented in~\cite{f16}, \cite{sqs17}, and~\cite{OrMorphism},
respectively.\footnote{Let us emphasize that such operators are in general not unique. For instance, there are known solutions~$\Diamond$ besides the subtrahend in the right\/-\/hand side of~\eqref{EqDiffToSchoutenWith}. E.g., the operator~$\Diamond$ in~\eqref{EqDiamond} for the orgraph sum~$\Or(\boldsymbol{\gamma}_5)(\cP)$ is different from the factorisation found in~\cite{sqs17}. Likewise, two linearly independent solutions of~\eqref{EqDiamond} for the tetrahedral flow $\Or(\boldsymbol{\gamma}_3)(\cP)$ are built in~\cite{f16}.%
}
\end{example}

\begin{state}[proved in App.~\protect\ref{AppLieAlgMor}]\label{LemmaCommutatorGraSym}
For a given Poisson bi\/-\/vector $\cP$, the graph orientation mapping, 
\[\Or(\cdot)(\cP)\colon \ker \Id{\bigr|}_{(n,2n-2)} \ni \gamma \mapsto \cQ(\cP) \in \ker \partial_\cP,\] 
is a Lie superalgebra morphism that takes the bracket of two cocycles in bi\/-\/grading $(n,2n-2)$ to the commutator $[\frac{d}{d\varepsilon_1}, \frac{d}{d\varepsilon_2}](\cP)$ of two symmetries $\frac{d}{d\varepsilon_i}(\cP) = \cQ_i(\cP)$.\footnote{%
By Brown 
\cite{Brown2012}, the commutator does in general --\,for a generic Poisson structure~$\cP$\,-- not vanish for Willwacher's odd\/-\/sided wheel cocycles.} 
\end{state}

\section{Kontsevich orgraphs: The count of multiplicicities and signs}\label{SecMultSignRules}
\noindent%
In this section we derive the rule for 
calculation of 
orgraph multiplicities and matching 
the signs of oriented Kontsevich graphs.
This 
defining property completely determines the evaluation $\Or(\gamma)(\cP)$ of the orientation morphism for cocycles~$\gamma\in\ker\Id$ at (tuples of) Poisson bivectors~$\cP$.

\begin{convent}
In the sequel, for a given graph~$\gamma_a$ in a cocycle $\sum_a c_a\cdot\gamma_a=\gamma\in\ker\Id$,
sums are taken over the big set of $2^{\#\sE(\gamma_a)}$ of ways to orient edges of~$\gamma_a$ by using the operator~$\vbDelta$.
\end{convent}

Not every such way to orient edges would yield a Kontsevich orgraph built of $n$ wedges (hence, having $n$ internal vertices, $2n$ arrow edges formed by $2n-2$ old edges of~$\gamma_a$ and $2$ new edges $S_0$, $S_1$ to the new sink vertices~$\mathsf{0}$, $\mathsf{1}$ with their content~$f$, $g$). 
But the Kontsevich orgraphs are selected automatically whenever
a copy of bi\/-\/vector~$\cP=\tfrac{1}{2} P^{ij}(\boldsymbol{x})\,\xi_i\xi_j$ is placed in each internal vertex of every orgraph, so that every vertex of~$\gamma_a$ is the arrowtail for exactly two oriented edges.

\begin{convent}
Every Kontsevich orgraph on $n$ internal vertices and $2n$ edges is 
encoded by the ordered (using an arbitrary fixed labelling of internal vertices) list of ordered (w.r.t.\ Left~$<$ Right at a vertex) pairs of target \textsl{vertices} for the respective outgoing edges.
\end{convent}

\begin{define}\label{DefMustache}
A Kontsevich orgraph (built of wedges) is a \emph{$\Lambda$-\/shaped} orgraph if one vertex is the arrowtail of both edges directed to the sinks. Otherwise, a Kontsevich orgraph --\,in which the edges to sinks are issued from different 
vertices\,-- is called a \emph{$\Pi$-\/shaped} orgraph (see Fig.~\ref{FigTetra}).
\end{define}

\begin{figure}[htb]
\begin{center}
{\unitlength=1mm
\begin{picture}(30,45)(-15,-25)
\put(-16,18){$(a)$}
\put(-15,0){\vector(1,0){30}}
\put(0,20){\vector(-3,-4){15}}
\put(15,0){\vector(-3,4){15}}
\put(0,20){\vector(0,-1){29}}
\put(-15,0){\vector(3,-2){14}}
\put(15,0){\vector(-3,-2){14}}
\put(0,-10){\vector(3,-4){7.5}}
\put(0,-10){\vector(-3,-4){7.5}}
\put(-16,-1){\llap{{\footnotesize 3}}}
\put(16,-1){{{\footnotesize 4}}}
\put(-0.75,-13.5){{{\footnotesize 2}}}
\put(-8,-21){\llap{{\footnotesize 0}}}
\put(8,-21){{{\footnotesize 1}}}
\put(-0.5,20.5){{{\footnotesize 5}}}
\put(-9,-7.5){\llap{\small I}}
\put(9,-7.5){\small II}
\put(0.75,6){\small III}
\put(-8,0.75){\small IV}
\put(-10,7){\llap{\small V}}
\put(10,7){\small VI}
\put(-4,-15){\llap{\small $S_0$}}
\put(4,-15){\small $S_1$}
\end{picture}
\qquad\raisebox{24mm}{$-\mathbf{3}\cdot$}\qquad
\begin{picture}(30,45)(-15,-25)
\put(-16,18){$(b)$}
\put(-15,0){\vector(1,0){30}}
\put(0,20){\vector(-3,-4){15}}
\put(15,0){\vector(-3,4){15}}
\put(0,20){\vector(0,-1){29}}
\put(15,0){\vector(-3,-2){15}}
\put(0,-10){\vector(-3,2){14}}
\put(-15,0){\vector(0,-1){10}}
\put(0,-10){\vector(0,-1){10}}
\put(-16,-1){\llap{{\footnotesize 3}}}
\put(16,-1){{{\footnotesize 4}}}
\put(0.75,-12.5){{{\footnotesize 2}}}
\put(-0.75,-21){\llap{{\footnotesize 0}}}
\put(-14.25,-11){{{\footnotesize 1}}}
\put(-0.5,20.5){{{\footnotesize 5}}}
\put(-9,-7.5){\llap{\small I}}
\put(9,-7.5){\small II}
\put(0.75,6){\small III}
\put(-8,0.75){\small IV}
\put(-10,7){\llap{\small V}}
\put(10,7){\small VI}
\put(0.5,-17){{\small $S_0$}}
\put(-15.5,-7){\llap{{\small $S_1$}}}
\end{picture}
\qquad\raisebox{24mm}{$-\mathbf{3}\cdot$}\qquad
\begin{picture}(30,45)(-15,-25)
\put(-16,18){$(c)$}
\put(-15,0){\vector(1,0){30}}
\put(-15,0){\vector(3,4){15}}
\put(0,20){\vector(0,-1){29}}
\put(0,20){\vector(3,-4){15}}
\put(15,0){\vector(-3,-2){15}}
\put(0,-10){\vector(-3,2){14}}
\put(0,-10){\vector(0,-1){10}}
\put(15,0){\vector(0,-1){10}}
\put(-16,-1){\llap{{\footnotesize 3}}}
\put(16,-1){{{\footnotesize 4}}}
\put(-0.75,-13){\llap{{\footnotesize 2}}}
\put(-0.75,-21){\llap{{\footnotesize 0}}}
\put(14.25,-11){\llap{{\footnotesize 1}}}
\put(-0.5,20.5){{{\footnotesize 5}}}
\put(-9,-7.5){\llap{\small I}}
\put(9,-7.5){\small II}
\put(0.75,6){\small III}
\put(-8,0.75){\small IV}
\put(-10,7){\llap{\small V}}
\put(10,7){\small VI}
\put(0.5,-17){{\small $S_0$}}
\put(16,-7.5){\small $S_1$}
\end{picture}
\ \raisebox{24mm}{.}
}
\end{center}
\caption{The right\/-\/hand side $\cQ_{1:\frac{6}{2}}$ of the tetrahedral flow is encoded by the $\Lambda$-\/shaped and two $\Pi$-\/shaped orgraphs; the edge ordering in the wedges is~\eqref{EqTetra}, cf.\ Example~\ref{ExTetraLookAtEdges} below.}\label{FigTetra}
\end{figure}
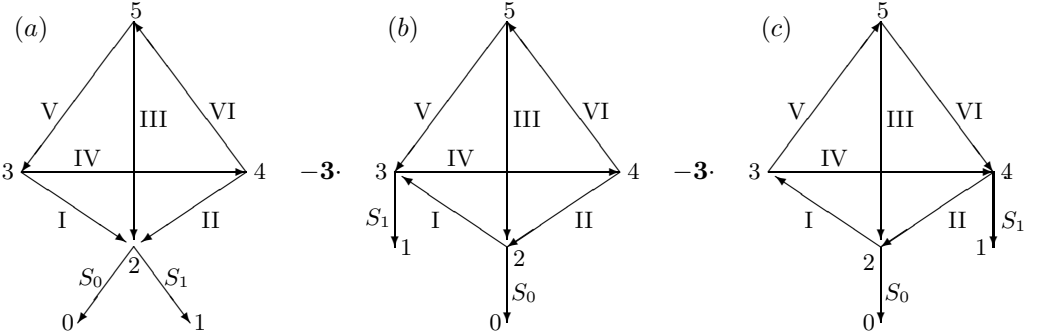

\begin{example}\label{ExTetraFlow}
In the tetrahedral flow 
\begin{equation}\label{EqTetra}
\cQ_{1:\frac{6}{2}}=
1 \cdot( \mathsf{0},\mathsf{1} ; \mathsf{2},\mathsf{4} ; \mathsf{2},\mathsf{5} ; \mathsf{2},\mathsf{3})
-3 \cdot( \mathsf{0},\mathsf{3} ; \mathsf{1},\mathsf{4} ; \mathsf{2},\mathsf{5} ; \mathsf{2},\mathsf{3}  + 
   \mathsf{0},\mathsf{3} ; \mathsf{4},\mathsf{5} ; \mathsf{1},\mathsf{2} ; \mathsf{2},\mathsf{ 4} )   
,
\end{equation}
the $\Lambda$-\/shaped orgraph and the $\Pi$-\/shaped orgraph, skew\/-\/symmetrised over its sinks, occur in proportion $8:24=1:3$ (see~\cite{f16} and~\cite{Bourbaki2017}).%
\footnote{
We verify in App.~\ref{AppRotate} that the third graph is equal to minus the second term with edges $S_0\rightleftarrows S_1$ interchanged.}
\end{example}

\begin{rem}
Each encoding of a way to orient a graph~$\gamma_a$ to a Kontsevich orgraph on $2n$ edges yields a permutation of the ordered set 
$\sE(\gamma_a)\wedge S_0\wedge S_1 = + S_0\wedge S_1\wedge\sE(\gamma_a)$ of these 
edges. 
Namely, the permutation is encoded by the ordered (using any given labelling of internal vertices) list of ordered (w.r.t.\ Left~$<$ Right at a vertex) pairs of outgoing \textsl{edges}.
\end{rem}

\begin{convent}
We let the default, occurring with `$+$' sign, ordering of oriented edges in a Kontsevich orgraph under study be $S_0\prec S_1\prec \sE(\gamma_a) = S_0\prec S_1\prec I\prec II\prec\cdots$.
\end{convent}

\begin{example}\label{ExTetraLookAtEdges}
The tetrahedral graph flow is 
\begin{multline*}
\cQ_{1:\frac{6}{2}}= (+1)\cdot\left(
\begin{smallmatrix}
\mathsf{0 } & \mathsf{1 } \\
  S_0 & S_1  
\end{smallmatrix}\ %
\begin{smallmatrix}
\mathsf{2 } & \mathsf{4 } \\
  I & IV  
\end{smallmatrix}\ %
\begin{smallmatrix}
\mathsf{2 } & \mathsf{5 } \\
  II & VI  
\end{smallmatrix}\ %
\begin{smallmatrix}
\mathsf{2 } & \mathsf{3 } \\
  III & V  
\end{smallmatrix}
\right) 
\\
{}- 3\cdot\left[
\left(
\begin{smallmatrix}
\mathsf{0 } & \mathsf{3 } \\
  S_0 & I  
\end{smallmatrix}\ %
\begin{smallmatrix}
\mathsf{1 } & \mathsf{4 } \\
  S_1 & IV  
\end{smallmatrix}\ %
\begin{smallmatrix}
\mathsf{2 } & \mathsf{5 } \\
  II & VI  
\end{smallmatrix}\ %
\begin{smallmatrix}
\mathsf{2 } & \mathsf{3 } \\
  III & V  
\end{smallmatrix}
\right)
+
\left(
\begin{smallmatrix}
\mathsf{0 } & \mathsf{3 } \\
  S_0 & I  
\end{smallmatrix}\ %
\begin{smallmatrix}
\mathsf{4 } & \mathsf{5 } \\
  IV & V  
\end{smallmatrix}\ %
\begin{smallmatrix}
\mathsf{1 } & \mathsf{2 } \\
  S_1 & II  
\end{smallmatrix}\ %
\begin{smallmatrix}
\mathsf{2 } & \mathsf{4 } \\
  III & VI  
\end{smallmatrix}
\right)
\right].
\end{multline*}
The parity sign of edge permutation in the $\Lambda$-\/shaped graph is $(-)^{4}=(+)$, here we count the transpositions w.r.t.\ $S_0\prec S_1\prec I\prec\ldots\prec VI$; the permutations of edges in the second and third, $\Pi$-\/shaped graphs are parity\/-\/odd: $(-)^{5}=(-)=(-)^{7}$, respectively.
\end{example}

Using the set of all the $2^{\#\sE(\gamma_a)}$ orientations of edges in a single graph $(\gamma_a$,\ $\sE(\gamma_a))$ from a cocycle~$\gamma$, select all topologically isomorphic Kontsevich orgraphs.

\begin{theor}\label{ThSignsOnlyFromPermutEdges}
Every orgraph~$\Gamma$, isomorphic to a Kontsevich orgraph~$\Gamma_0$,
acquires under~\eqref{EqOr} the sign
\begin{equation}\label{EqEpsilon}
\sign(\Gamma) = (-)^{\sigma_E} \cdot \sign(\Gamma_0),
\end{equation}
where the permutation~$\sigma_E\colon \Edge(\Gamma)\simeq\Edge(\Gamma_0)$ of oriented edges is induced by the orgraph isomorphism
$\sigma\colon \Gamma\simeq\Gamma_0$.
\end{theor}

This is the rule for the count of orgraph multiplicities and signs.
For instance, the multiplicity $+8$ of the first Kontsevich orgraph in the tetrahedral flow (cf.~\cite{Ascona96,Bourbaki2017} and~\cite{f16}) is obtained in Example~\ref{ExMustache8} at the end of this section.
Examples illustrating how rule~\eqref{EqEpsilon} works in the various counts of signs will be given in section~\ref{SecRuleOfSignsIntrinsic} (see p.~\pageref{ExTetra} below, as well as Example~\ref{ExTetraLookAtEdges} earlier in this section).\footnote{The secret confined in Theorem~\ref{ThSignsOnlyFromPermutEdges} is that the calculation of edge permutation parities is in fact redundant for finding the signs in front of the Kontsevich orgraphs; this will be seen in the next section.}

\begin{proof}
Represent the $\alpha$th copy of Poisson bi\/-\/vector~$\cP_{(\alpha)}$ using $\xi_{p(\alpha)}\xi_{q(\alpha)}\cdot \tfrac{1}{2} P_{(\alpha)}^{p(\alpha)\,q(\alpha)}(\boldsymbol{x})$, here $1\leqslant\alpha\leqslant n$, and collect the 
set of parity\/-\/even pairs $(\xi\xi)_{\alpha}$ to the left of the product of bi\/vector coefficients. The edge orienting operator~$\vbDelta$ acts from the left on the string $(\xi\xi)_1\cdot\ldots\cdot(\xi\xi)_n$; the topology of unoriented graph~$\gamma_a$ and the choice of $\vec{\dd}/\dd x^\nu_{(i)}\otimes\vec{\dd}/\dd\xi_\nu^{(j)}$ versus
$\vec{\dd}/\dd x^\nu_{(j)}\otimes\vec{\dd}/\dd\xi_\nu^{(i)}$ at each pair $(ij)$
specify the $n$th degree differential monomial in coefficients~$P_{(\alpha)}^{p(\alpha)\,q(\alpha)}$. (Note that the index 
$\nu$ for every edge $e_{ij}$ is a summation index.) 

Under the orgraph isomorphism~$\sigma$, which unshuffles the vertex pairs~$(\xi\xi)_\alpha$ by an even\/-\/parity permutation of the letters~$\xi$, the parity\/-\/odd edge operators~$\Delta_{ij}$ and the two new edges~$S_0$, $S_1$ are permuted by~$\sigma_E$.
This permutation of edges corresponds to a permutation of the parity\/-\/odd letters~$\xi$ that indicate the tails of those edges. 
Therefore, the two Kontsevich orgraph encodings differ by the sign factor~$(+)\cdot (-)^{\sigma_E}=(-)^{\sigma_E}$.
\end{proof}

\begin{cor}\label{CorSkewGuaranteed}
A skew\/-\/symmetry w.r.t.\ the sinks is guaranteed for bi\/-\/vector orgraphs. Indeed, for orgraphs with swapped new edges $S_0$ and~$S_1$ issued to the sinks $f$, $g$ from some vertex (or two distinct vertices) of~$\gamma_a$, the sign of transposition~$S_0\rightleftarrows S_1$ balances the two types of orgraphs: with $S_0$ to~$f$ and $S_1$ to~$g$ against orgraphs with the edge $S_1$ now issued to $g$ from the old source of~$S_0$, and the new edge $S_0$ issued to $f$ from the old source of the old edge~$S_1$.
\end{cor}

\begin{lemma}
The sign 
of edge permutation 
in~\eqref{EqEpsilon} trivializes the contribution from any \emph{zero} unoriented graph~$\cZ$ to the sum over all possible ways to orient its edges.
\end{lemma}

\begin{proof}
Run over the entire sum of $2^{\#\sE(\cZ)}$ ways to orient edges in a zero graph~$\cZ$, and select all admissible not identically zero differential monomials similar to a given one, \textsl{i.e.}\ pick all Kontsevich orgraphs~$\Gamma$ isomorphic to a given one (denote it by~$\Gamma_0$). Each isomorphism $\Gamma\simeq\Gamma_0$ represents an 
element of the automorphism group~$\Aut(\cZ)$.
By the definition of zero unoriented graph, the subgroup $H\unlhd\mathbb{S}_{2n-2}$ of unoriented edge permutations --\,under the action of the group~$\Aut(\cZ)$\,-- contains at least one parity\/-\/odd element~$\sigma_E$. Therefore, the numbers of parity\/-\/even and parity\/-\/odd elements in~$H$ coincide.
This implies that these $\#\Aut(\cZ)$ elements cancel in disjoint pairs, so that the contribution of each topological profile in $\Or(\cZ)(\cP$,\ $\ldots$,\ $\cP)$ comes with zero coefficient.
\end{proof}

\begin{rem}
The wedge orderings Left~$<$ Right of outgoing edges at internal vertices of the Kontsevich orgraphs under study
play no role in 
the count of multiplicities 
and 
signs\,! 
\end{rem}

\begin{example}\label{ExMustache8}
Indeed, 
were 
the Left~$<$ Right oriented edge orderings contributing with a `$-$' sign factor per vertex whenever Right~$\prec$ Left in the edge ordering $S_0\prec S_1 \prec I\prec II\prec\ldots$, 
the $\Lambda$-\/shaped graph of differential orders $(3,1,1,1)$ in the tetrahedral flow $\Or(\boldsymbol{\gamma}_3)(\cP)$ 
would accumulate the wrong coefficient $4-4=0$ instead of the true value $4+4=8$ in the balance $8:24$ for the linear combination $\cQ_{1:\frac{6}{2}}(\cP)$ of the Kontsevich graphs (see Fig.~\ref{FigLambdaTetraAtoH} 
\begin{figure}[htb]
\begin{center}
{\unitlength=1mm
\begin{picture}(30,45)(-15,-25)
\put(-14,18){\llap{$\{a\}$}}
\put(-15,0){\vector(1,0){30}}
\put(0,20){\vector(-3,-4){15}}
\put(15,0){\vector(-3,4){15}}
\put(0,20){\vector(0,-1){29}}
\put(-15,0){\vector(3,-2){14}}
\put(15,0){\vector(-3,-2){14}}
\put(0,-10){\vector(3,-4){7.5}}
\put(0,-10){\vector(-3,-4){7.5}}
\put(-16,-1){\llap{{\footnotesize 2}}}
\put(16,-1){{{\footnotesize 3}}}
\put(-0.75,-13.5){{{\footnotesize 1}}}
\put(-0.5,20.5){{{\footnotesize 4}}}
\put(-8,-21){\llap{{\footnotesize $f$}}}
\put(8,-21){{{\footnotesize $g$}}}
\put(-9,-7.5){\llap{\small I}}
\put(9,-7.5){\small II}
\put(0.75,6){\small III}
\put(-8,0.75){\small IV}
\put(-10,7){\llap{\small V}}
\put(10,7){\small VI}
\put(-4,-15){\llap{\small $S_0$}}
\put(4,-15){\small $S_1$}
\end{picture}
\qquad\qquad
\begin{picture}(30,45)(-15,-25)
\put(-14,18){\llap{$\{b\}$}}
\put(15,0){\vector(-1,0){30}}
\put(-15,0){\vector(3,4){15}}
\put(0,20){\vector(3,-4){15}}
\put(0,20){\vector(0,-1){29}}
\put(-15,0){\vector(3,-2){14}}
\put(15,0){\vector(-3,-2){14}}
\put(0,-10){\vector(3,-4){7.5}}
\put(0,-10){\vector(-3,-4){7.5}}
\put(-8,-21){\llap{{\footnotesize $f$}}}
\put(8,-21){{{\footnotesize $g$}}}
\put(-9,-7.5){\llap{\small I}}
\put(9,-7.5){\small II}
\put(0.75,6){\small III}
\put(-8,0.75){\small IV}
\put(-10,7){\llap{\small V}}
\put(10,7){\small VI}
\put(-4,-15){\llap{\small $S_0$}}
\put(4,-15){\small $S_1$}
\end{picture}
\qquad\qquad
\begin{picture}(30,45)(-15,-25)
\put(-14,18){\llap{$\{c\}$}}
\put(-15,0){\vector(3,-2){14}}
\put(0,-10){\vector(3,2){14}}
\put(0,-10){\vector(0,1){29}}
\put(-15,0){\vector(1,0){30}}
\put(0,20){\vector(-3,-4){15}}
\put(0,20){\vector(3,-4){15}}
\put(15,0){\vector(-1,-4){2.5}}
\put(15,0){\vector(1,-4){2.5}}
\put(12,-12){\llap{{\footnotesize $f$}}}
\put(17.5,-12){{{\footnotesize $g$}}}
\put(-9,-7.5){\llap{\small I}}
\put(9,-7.5){\small II}
\put(0.75,6){\small III}
\put(-8,0.75){\small IV}
\put(-10,7){\llap{\small V}}
\put(10,7){\small VI}
\end{picture}
\\[5pt]
\begin{picture}(30,45)(-15,-25)
\put(-14,18){\llap{$\{d\}$}}
\put(0,-10){\vector(-3,2){14}}
\put(0,-10){\vector(3,2){14}}
\put(0,20){\vector(0,-1){29}}
\put(-15,0){\vector(1,0){30}}
\put(-15,0){\vector(3,4){15}}
\put(0,20){\vector(3,-4){15}}
\put(15,0){\vector(-1,-4){2.5}}
\put(15,0){\vector(1,-4){2.5}}
\put(12,-12){\llap{{\footnotesize $f$}}}
\put(17.5,-12){{{\footnotesize $g$}}}
\put(-9,-7.5){\llap{\small I}}
\put(9,-7.5){\small II}
\put(0.75,6){\small III}
\put(-8,0.75){\small IV}
\put(-10,7){\llap{\small V}}
\put(10,7){\small VI}
\end{picture}
\qquad\qquad
\begin{picture}(30,45)(-15,-25)
\put(-14,18){\llap{$\{e\}$}}
\put(0,-10){\vector(-3,2){14}}
\put(0,-10){\vector(3,2){14}}
\put(0,20){\vector(0,-1){29}}
\put(15,0){\vector(-1,0){30}}
\put(0,20){\vector(-3,-4){15}}
\put(15,0){\vector(-3,4){15}}
\put(-15,0){\vector(-1,-4){2.5}}
\put(-15,0){\vector(1,-4){2.5}}
\put(-17.75,-12){\llap{{\footnotesize $f$}}}
\put(-12.25,-12){{{\footnotesize $g$}}}
\put(-9,-7.5){\llap{\small I}}
\put(9,-7.5){\small II}
\put(0.75,6){\small III}
\put(-8,0.75){\small IV}
\put(-10,7){\llap{\small V}}
\put(10,7){\small VI}
\end{picture}
\qquad\qquad
\begin{picture}(30,45)(-15,-25)
\put(-14,18){\llap{$\{f\}$}}
\put(0,-10){\vector(-3,2){14}}
\put(15,0){\vector(-3,-2){14}}
\put(0,-10){\vector(0,1){29}}
\put(15,0){\vector(-1,0){30}}
\put(0,20){\vector(-3,-4){15}}
\put(0,20){\vector(3,-4){15}}
\put(-15,0){\vector(-1,-4){2.5}}
\put(-15,0){\vector(1,-4){2.5}}
\put(-17.75,-12){\llap{{\footnotesize $f$}}}
\put(-12.25,-12){{{\footnotesize $g$}}}
\put(-9,-7.5){\llap{\small I}}
\put(9,-7.5){\small II}
\put(0.75,6){\small III}
\put(-8,0.75){\small IV}
\put(-10,7){\llap{\small V}}
\put(10,7){\small VI}
\end{picture}
\\[15pt]
\begin{picture}(30,30)(-15,-10)
\put(-14,18){\llap{$\{g\}$}}
\put(0,-10){\vector(-3,2){14}}
\put(15,0){\vector(-3,-2){14}}
\put(0,-10){\vector(0,1){29}}
\put(-15,0){\vector(1,0){30}}
\put(-15,0){\vector(3,4){15}}
\put(15,0){\vector(-3,4){15}}
\put(0,20){\vector(-1,4){2.5}}
\put(0,20){\vector(1,4){2.5}}
\put(-3,28){\llap{{\footnotesize $f$}}}
\put(2.75,28){{{\footnotesize $g$}}}
\put(-9,-7.5){\llap{\small I}}
\put(9,-7.5){\small II}
\put(0.75,6){\small III}
\put(-8,0.75){\small IV}
\put(-10,7){\llap{\small V}}
\put(10,7){\small VI}
\end{picture}
\qquad\qquad
\begin{picture}(30,30)(-15,-10)
\put(-14,18){\llap{$\{h\}$}}
\put(-15,0){\vector(3,-2){14}}
\put(0,-10){\vector(3,2){14}}
\put(0,-10){\vector(0,1){29}}
\put(15,0){\vector(-1,0){30}}
\put(-15,0){\vector(3,4){15}}
\put(15,0){\vector(-3,4){15}}
\put(0,20){\vector(-1,4){2.5}}
\put(0,20){\vector(1,4){2.5}}
\put(-3,28){\llap{{\footnotesize $f$}}}
\put(2.75,28){{{\footnotesize $g$}}}
\put(-9,-7.5){\llap{\small I}}
\put(9,-7.5){\small II}
\put(0.75,6){\small III}
\put(-8,0.75){\small IV}
\put(-10,7){\llap{\small V}}
\put(10,7){\small VI}
\end{picture}
}
\end{center}
\caption{The eight $\Lambda$-\/shaped ways to orient the tetrahedron: $\{a\}\div\{h\}$ yield $+8$.}\label{FigLambdaTetraAtoH}
\end{figure}
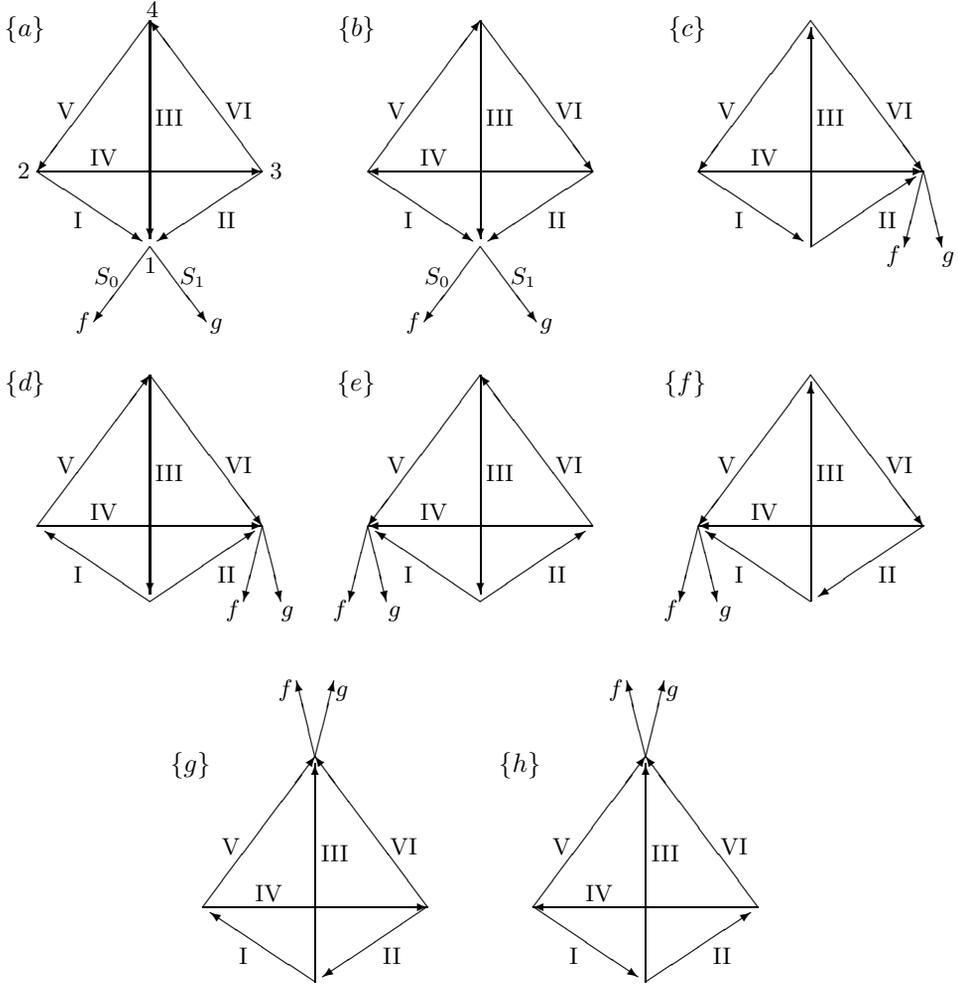
and Table~\ref{TabTetraMustacheEight}). %
\begin{table}[htb]
\caption{Permutations of six edges $I$, $\ldots$, $VI$ in the $\Lambda$-\/shaped orientations of the tetrahedron%
.}\label{TabTetraMustacheEight}
\begin{center}
\begin{tabular}{rccccccccc}
\hline
  & & & & & & & $(-)^{\sigma_E}$ & $\#(L>R)$ & $(-)^{\#(L>R)}$  \\[1pt]
\hline
$\{a\}$ & I & II & III & IV & V & VI & $+$ & 0 & $+$ \\
$\{b\}$ & I & III & II & V & IV & VI & $+$ & 0 & $+$ \\
$\{c\}$ & II & VI & IV & III & I & V & $+$ & 2 & $+$ \\
$\{d\}$ & II & IV & VI & I & III & V & $+$ & 2 & $+$ \\
$\{e\}$ & IV & I & V & II & VI & III & $+$ & 1 & $-$ \\
$\{f\}$ & I & V & IV & III & II & VI & $+$ & 1 & $-$ \\
$\{g\}$ & V & VI & III & IV & I & II & $+$ & 3 & $-$ \\
$\{h\}$ & V & III & VI & I & IV & II & $+$ & 3 & $-$ \\[1pt]
\hline
\end{tabular}
\end{center}
\end{table}
For instance, $I^a\simeq V^h$ and $IV^a\simeq I^h$ so that in the orgraph $\{a\}$, $I^a<IV^a$, but under the orgraph isomorphism $\{a\}\simeq\{h\}$, one obtains the inversion $V^h<I^h$ at a vertex.\footnote{The full list of inversions is this: $\{c\}$ $V>VI$, $I>IV$; $\{d\}$ $I>II$, $III>VI$; $\{e\}$ $III>V$; $\{f\}$ $II>IV$; $\{g\}$ $IV>V$, $II>VI$, $I>III$; $\{h\}$ $I>V$, $II>III$, $IV>VI$.}
\end{example}

\section{The rule of signs in terms of Kontsevich orgraphs}\label{SecRuleOfSignsIntrinsic}
\noindent%
In this section we derive two rules which allow the matching of signs in front of Kontsevich's orgraphs by simply looking at these graphs (or their encodings), so that no calculation of parities for permutations of \emph{all} the edges is needed. The first rule is specific to $\Pi$-\/shaped orgraphs. The second rule describes the sign factors which are gained 
in the course of transitions $\Pi\rightleftarrows\Pi$, $\Lambda\rightleftarrows\Pi$, and $\Lambda\rightleftarrows\Lambda$ between the orgraphs of respective shapes, as long as they are taken from the set of all admissible ways to orient a given graph~$\gamma_a$ in a cocycle $\gamma=\sum_a c_a\cdot\gamma_a$. The work of both rules is illustrated using the tetrahedral and pentagon\/-\/wheel cocycle flows $\Or(\boldsymbol{\gamma}_3)(\cP)$ and~$\Or(\boldsymbol{\gamma}_5)(\cP)$ from~\cite{f16} and~\cite{sqs17}, respectively.

\begin{Rule}\label{RuleReverseSinksOrder}
A $\Pi$-\/shaped orgraph with ordered edge pairs $(S_0,A)(S_1,B)\cdots$ issued from two distinct vertices acquires under~\eqref{EqOr} the extra sign factor~$(-)$, compared with a graph with the ordered edge pairs $(S_0,B)(S_1,A)\cdots$, if $A\prec B$ in the edge ordering~$\sE(\gamma_a)$ of a graph to orient.
\end{Rule}

\begin{proof}
Indeed, $S_0\wedge S_1\wedge A\wedge B = -S_0\wedge A\wedge S_1\wedge B = (-)^2 S_0\wedge B\wedge S_1\wedge A = + S_1\wedge A\wedge S_0\wedge B$.
\end{proof} 

\begin{example}
Both $\Pi$-\/shaped graphs in the tetrahedral flow (see its encoding in Example~\ref{ExTetraLookAtEdges} in section~\ref{SecMultSignRules}) do acquire a sign factor by Rule~\ref{RuleReverseSinksOrder}.
\end{example}

\begin{define}
The \textsl{body} of a Kontsevich orgraph which is obtained by orienting~$\gamma_a$ in a cocycle~$\gamma$ is the set of oriented 
edges inherited from~$\gamma_a$, \textsl{i.e.}\ excluding the new edges~$S_i$ to the sinks.
\end{define}

\begin{Rule}\label{RuleRevers}
Let $\Gamma_1$ and~$\Gamma_2$ be two topologically nonisomorphic orgraphs which are obtained by orienting the same graph~$\gamma_a$ in a cocycle~$\gamma$.\footnote{%
For instance, such obviously are all the 
terms 
in the Kontsevich flow $\Or(\boldsymbol{\gamma}_3)(\cP)$ where the tetrahedron $\boldsymbol{\gamma}_3\in\ker\Id$ is oriented,
or the orgraphs which one obtains by orienting the pentagon wheel and the prism 
graph in the Kontsevich\/--\/Willwacher cocycle~$\boldsymbol{\gamma}_5$, cf.~\cite{JNMP2017,sqs17}.%
}
\begin{description}
\item[$\Pi\rightleftarrows\Pi$]
If both the orgraphs are $\Pi$-\/shaped, then
the sign in front of (the multiplicity of) the orgraph~$\Gamma_2$ is determined from such sign given by~\eqref{EqOr} for~$\Gamma_1$ by now using the formula
\begin{equation}\label{EqReverseInBody}
\sign(\Gamma_2)=(-)^{\#\left\{\text{\parbox{5.2cm}{\ reverses of arrows in the body as $\Gamma_1 \to \Gamma_2$}}\right\}
}\cdot\sign(\Gamma_1).
\end{equation}
\item[$\Lambda\rightleftarrows\Pi$]
Transitions $\Lambda\rightleftarrows\Pi$ yield the product of sign factors $(-)\times{}$formula~\eqref{EqReverseInBody}, \textsl{i.e.}\ the extra~$(-)$ is universal, distinguishing between the shapes.
\item[$\Lambda\rightleftarrows\Lambda$]
Same\/-\/shape transitions $\Lambda\rightleftarrows\Lambda$ acquire only the sign factor~\eqref{EqReverseInBody}.
\end{description}
In other words, the transition $\Lambda\rightleftarrows\Pi$ signals the sign factor $(-)$, and the number of body arrow reversals contributes in all cases.
\end{Rule}


\begin{proof}
\textbf{Case~$\Pi\rightleftarrows\Pi$.}
For the sake of clarity, assume at once that the edge operators~$\vec{\Delta}_{ij}$ corresponding to the edges whose orientation is \emph{not} reversed have already acted on the argument of two operators~$\vbDelta$ corresponding to the two graphs,~$\Gamma_1$ and~$\Gamma_2$.
There remain $\kappa$~edge operators acting on the product of $\kappa+2$ comultiples $\xi\cdots\xi$ times an even factor formed by the coefficients~$P^{pq}_{(\alpha)}(\bx)$ of bi\/-\/vector copies.

Consider the righmost operator~$\vec{\Delta}_{ij}$ from what remains; it is the sum $\vec{\dd}/\dd\xi^{\text{old}}_{\text{tail}} \otimes
\vec{\dd}/\dd x^{\text{old}}_{\text{head}}$ and
$\vec{\dd}/\dd\xi^{\text{old}}_{\text{head}} \otimes
\vec{\dd}/\dd x^{\text{old}}_{\text{tail}} =
\vec{\dd}/\dd\xi^{\text{new}}_{\text{tail}} \otimes
\vec{\dd}/\dd x^{\text{new}}_{\text{head}}$.
Because the derivatives $\vec{\dd}/\dd\bx$ have even parity, we focus on the choice of superderivation to orient the edge (resp., fix and then reverse its orientation). In the ordered string~$\xi\cdots\xi$, let us bring next to each other the symbols~$\xi_i$ and~$\xi_j$ from the copies~$\cP_{(i)}$ and~$\cP_{(j)}$ contained in the~$i$th and~$j$th vertices. It is obvious that the action by~$\vec{\dd}/\dd\xi$ on one such comultiple instead of the other creates the sign factor~$(-)$. Doing this $\kappa$~times counts the number of arrow reversions in the body of~$\Pi$-\/shaped orgraph, whence~$(-)^\kappa$.

\smallskip
\noindent\textbf{Case~$\Lambda\rightleftarrows\Pi$.}
To avoid an agglomeration of symbols, we omit the letters~$\xi$ and display their subscripts, thus indicating either which body edge it is (say $A$ or $B$, $A\prec B$) or where it goes to ($S_0\mathrel{{:}{=}}F$ to the argument~$f$ in the sink~$\mathsf{0}$ and $S_1\mathrel{{:}{=}}G$ to the argument~$g$ in the sink~$\mathsf{1}$). Remember that the edge letters $A$, $B$, $F$, and~$G$ are parity\/-\/odd by construction.

Without loss of generality, let us assume that in the string of $2n$~comultiples the four rightmost are,
\[
\begin{aligned}
\text{for the $\Lambda$-\/shaped orgraph:\qquad} &  A\:B\ F\:G,\\
\text{for the $\Pi$-\/shaped orgraph:\qquad} & A\:F\ B\:G.\\
\end{aligned}
\]
We see that $(A\:B)\:(F\:G) = - (A\:F)\:(B\:G)$, whence we obtain the sought\/-\/for universal sign factor~$(-)$ for any transitions between the different shapes~$\Lambda\rightleftarrows\Pi$ (see Examples~\ref{ExTetra} and~\ref{ExPenta} in what follows). Now, the count of body edge reversals goes exactly as before.

\smallskip
\noindent\textbf{Case~$\Lambda\rightleftarrows\Lambda$.}
There remains almost nothing to prove: in the above notation, we have that $A\:B\ F\:G = F\:G\ A\:B$, hence no extra sign factor is produced when the wedge of two edges directed to sinks is transported from one internal vertex to another.\footnote{%
This will presently be illustrated in Example~\ref{ExPentaMustache2} by using topologically nonisomorphic $\Lambda$-\/shaped orgraphs in the set of admissible orientations of the pentagon wheel in the Kontsevich\/--\/Willwacher cocycle~$\boldsymbol{\gamma}_5$.}
\end{proof}

\begin{example}\label{ExTetra}
Consider the r.-h.s.\ 
$\cQ_{1:\frac{6}{2}}=\Or(\boldsymbol{\gamma}_3)(\cP)$ of the Kontsevich tetrahedral flow,
\begin{multline*}
\cQ_{1:\frac{6}{2}}= (+1)\cdot\underbrace{ \left(
\begin{smallmatrix}
\mathsf{0 } & \mathsf{1 } \\
  S_0 & S_1  
\end{smallmatrix}\ %
\begin{smallmatrix}
\mathsf{2 } & \mathsf{4 } \\
  I & IV  
\end{smallmatrix}\ %
\begin{smallmatrix}
\mathsf{2 } & \mathsf{5 } \\
  II & VI  
\end{smallmatrix}\ %
\begin{smallmatrix}
\mathsf{2 } & \mathsf{3 } \\
  III & V  
\end{smallmatrix}
\right) }_{\text{$\Lambda$-\/shaped}}
\\
{}- 3\cdot\Bigl[
\underbrace{ \left(
\begin{smallmatrix}
\mathsf{0 } & \mathsf{3 } \\
  S_0 & I  
\end{smallmatrix}\ %
\begin{smallmatrix}
\mathsf{1 } & \mathsf{4 } \\
  S_1 & IV  
\end{smallmatrix}\ %
\begin{smallmatrix}
\mathsf{2 } & \mathsf{5 } \\
  II & VI  
\end{smallmatrix}\ %
\begin{smallmatrix}
\mathsf{2 } & \mathsf{3 } \\
  III & V  
\end{smallmatrix}
\right) }_{\text{minuend}}
+
\underbrace{ \left(
\begin{smallmatrix}
\mathsf{0 } & \mathsf{3 } \\
  S_0 & I  
\end{smallmatrix}\ %
\begin{smallmatrix}
\mathsf{4 } & \mathsf{5 } \\
  IV & V  
\end{smallmatrix}\ %
\begin{smallmatrix}
\mathsf{1 } & \mathsf{2 } \\
  S_1 & II  
\end{smallmatrix}\ %
\begin{smallmatrix}
\mathsf{2 } & \mathsf{4 } \\
  III & VI  
\end{smallmatrix}
\right) }_{\text{subtrahend}}
\Bigr].
\end{multline*}
Using Rules~\ref{RuleReverseSinksOrder} and~\ref{RuleRevers}, let us show why the sign which relates the $\Lambda$-\/shaped orgraph to the skew\/-\/symmetrisation of $\Pi$-\/shaped orgraph is equal to~$(-)$; the count of multiplicities, $8:24=1:3$, is standard.\footnote{The admissible $\Lambda$-\/shaped orientations of the tetrahedron are obtained by attaching the wedge~$S_0S_1$ to one of the four vertices and orienting the opposite face using one of two admissible ways, so that~$4\cdot2=8$. The $\Pi$-\/shaped Kontsevich graphs are obtained by selecting an edge from six of them, directing it in one of the two ways, and orienting the opposite edge also in one of two ways, whence $6\cdot2\cdot2=24$.}
\begin{itemize}
\item In the minuend, which is a $\Pi$-\/shaped orgraph, Rule~\ref{RuleReverseSinksOrder} 
contributes --\,for the edge pairs $(S_0\:I)\:(S_1\:IV)\cdots
$\,-- with the first factor~$(-)$.
\item In the course of transition $\Lambda\rightleftarrows\Pi$ to the minuend, one arrow in the body of orgraph is reversed (namely, it is the edge~$I$ bridging the 
edges~$S_0$ and~$S_1$ issued from the vertices~$\mathsf{2}$ and~$\mathsf{3}$), whence another minus sign, $(-)=(-)^1$.
\item The transition $\Lambda\rightleftarrows\Pi$ itself contributes with a universal sign~$(-)$, see Rule~\ref{RuleRevers} again.
\end{itemize}
In total, we accumulate the sign factor~$(-)\cdot(-)\cdot(-)=(-)$, which indeed is the sign that relates the skew\/-\/symmetric orgraphs in the flow~$\dot{\cP}=\Or(\boldsymbol{\gamma}_3)(\cP)$.
\end{example}

\begin{example}\label{ExPenta}
Consider two $\Lambda$-\/shaped terms and a $\Pi$-\/shaped term --\,in Fig.~\ref{FigPenta}\,--
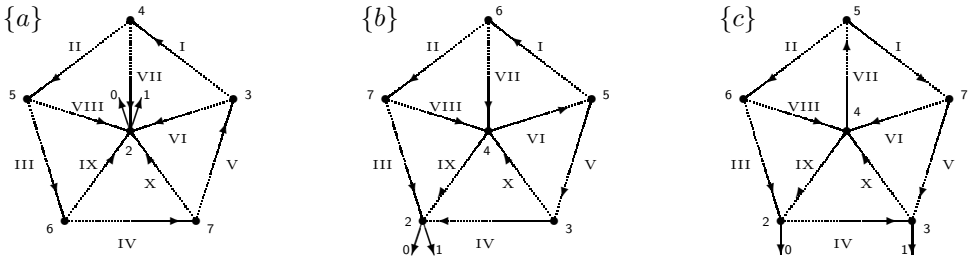
\begin{figure}[htb]
\centerline{
\raisebox{0pt}[72pt][56pt]{
{\unitlength=0.5mm
\begin{picture}(58,53)(-33.5,-5)
\put(-34,28){$\{a\}$}
\put(27.5,8.5){\circle*{2}}
\put(0,29.5){\circle*{2}}
\put(-27.5,8.5){\circle*{2}}
\put(-17.5,-23.75){\circle*{2}}
\put(17.5,-23.75){\circle*{2}}
\qbezier[70](27.5,8.5)(0,29.5)(0,29.5)
\qbezier[70](0,29.5)(-27.5,8.5)(-27.5,8.5)
\qbezier[70](-27.5,8.5)(-17.5,-23.75)(-17.5,-23.75)
\qbezier[70](-17.5,-23.75)(17.5,-23.75)(17.5,-23.75)
\qbezier[70](17.5,-23.75)(27.5,8.5)(27.5,8.5)
\put(0,5){\vector(0,-1){0}}
\put(-6,2){\vector(3,-1){0}}
\put(6,2){\vector(-3,-1){0}}
\put(-4,-6){\vector(2,3){0}}
\put(4,-6){\vector(-2,3){0}}
\put(6,25){\vector(-4,3){0}}
\put(-21.5,13){\vector(-4,-3){0}}
\put(-19.5,-16.75){\vector(1,-3){0}}
\put(13.5,-23.75){\vector(1,0){0}}
\put(25.5,2.5){\vector(1,3){0}}
\put(0,0){\circle*{2}}
\qbezier[60](27.5,8.5)(0,0)(0,0)
\qbezier[60](0,29.5)(0,0)(0,0)
\qbezier[60](-27.5,8.5)(0,0)(0,0)
\qbezier[60](-17.5,-23.75)(0,0)(0,0)
\qbezier[60](17.5,-23.75)(0,0)(0,0)
\put(30.5,8.7){{\tiny\textsf{3}}}
\put(2,31){{\tiny\textsf{4}}}
\put(-30.5,8.7){\llap{{\tiny\textsf{5}}}}
\put(-20.5,-27){\llap{{\tiny\textsf{6}}}}
\put(20.5,-27){{\tiny\textsf{7}}}
\put(-1.35,-6.25){{\tiny\textsf{2}}}
\put(13,21.5){{\tiny I}}
\put(-13,21.5){\llap{{\tiny II}}}
\put(-25.5,-10){\llap{{\tiny III}}}
\put(-3.5,-31){{\tiny IV}}
\put(25.5,-10){{\tiny V}}
\put(10,-3){{\tiny VI}}
\put(1.5,13.5){{\tiny VII}}
\put(-7,5.5){\llap{{\tiny VIII}}}
\put(-8.5,-10){\llap{{\tiny IX}}}
\put(7,-15){\llap{{\tiny X}}}
\put(0,0){\vector(-1,3){3}}
\put(0,0){\vector(1,3){3}}
\put(-3.5,8.7){\llap{\tiny\textsf{0}}}
\put(3.5,8.7){{\tiny\textsf{1}}}
\end{picture}
}
}
\qquad\qquad
\raisebox{0pt}[72pt][56pt]{
{\unitlength=0.5mm
\begin{picture}(58,53)(-33.5,-5)
\put(-34,28){$\{b\}$}
\put(27.5,8.5){\circle*{2}}
\put(0,29.5){\circle*{2}}
\put(-27.5,8.5){\circle*{2}}
\put(-17.5,-23.75){\circle*{2}}
\put(17.5,-23.75){\circle*{2}}
\qbezier[70](27.5,8.5)(0,29.5)(0,29.5)
\qbezier[70](0,29.5)(-27.5,8.5)(-27.5,8.5)
\qbezier[70](-27.5,8.5)(-17.5,-23.75)(-17.5,-23.75)
\qbezier[70](-17.5,-23.75)(17.5,-23.75)(17.5,-23.75)
\qbezier[70](17.5,-23.75)(27.5,8.5)(27.5,8.5)
\put(0,5){\vector(0,-1){0}}
\put(-6,2){\vector(3,-1){0}}
\put(21.5,6.5){\vector(3,1){0}}
\put(-13.5,-17.75){\vector(-2,-3){0}}
\put(4,-6){\vector(-2,3){0}}
\put(6,25){\vector(-4,3){0}}
\put(-21.5,13){\vector(-4,-3){0}}
\put(-19.5,-16.75){\vector(1,-3){0}}
\put(-13.5,-23.75){\vector(-1,0){0}}
\put(19.5,-17.75){\vector(-1,-3){0}}
\put(0,0){\circle*{2}}
\qbezier[60](27.5,8.5)(0,0)(0,0)
\qbezier[60](0,29.5)(0,0)(0,0)
\qbezier[60](-27.5,8.5)(0,0)(0,0)
\qbezier[60](-17.5,-23.75)(0,0)(0,0)
\qbezier[60](17.5,-23.75)(0,0)(0,0)
\put(30.5,8.7){{\tiny\textsf{5}}}
\put(2,31){{\tiny\textsf{6}}}
\put(-30.5,8.7){\llap{{\tiny\textsf{7}}}}
\put(-20.5,-25){\llap{{\tiny\textsf{2}}}}
\put(20.5,-27){{\tiny\textsf{3}}}
\put(-1.35,-6.25){{\tiny\textsf{4}}}
\put(13,21.5){{\tiny I}}
\put(-13,21.5){\llap{{\tiny II}}}
\put(-25.5,-10){\llap{{\tiny III}}}
\put(-3.5,-31){{\tiny IV}}
\put(25.5,-10){{\tiny V}}
\put(10,-3){{\tiny VI}}
\put(1.5,13.5){{\tiny VII}}
\put(-7,5.5){\llap{{\tiny VIII}}}
\put(-8.5,-10){\llap{{\tiny IX}}}
\put(7,-15){\llap{{\tiny X}}}
\put(-17.5,-23.75){\vector(-1,-3){3}}
\put(-17.5,-23.75){\vector(1,-3){3}}
\put(-21,-32.45){\llap{\tiny\textsf{0}}}
\put(-14,-32.45){{\tiny\textsf{1}}}
\end{picture}
}
}
\qquad\qquad
\raisebox{0pt}[72pt][56pt]{
{\unitlength=0.5mm
\begin{picture}(58,53)(-33.5,-5)
\put(-34,28){$\{c\}$}
\put(27.5,8.5){\circle*{2}}
\put(0,29.5){\circle*{2}}
\put(-27.5,8.5){\circle*{2}}
\put(-17.5,-23.75){\circle*{2}}
\put(17.5,-23.75){\circle*{2}}
\qbezier[70](27.5,8.5)(0,29.5)(0,29.5)
\qbezier[70](0,29.5)(-27.5,8.5)(-27.5,8.5)
\qbezier[70](-27.5,8.5)(-17.5,-23.75)(-17.5,-23.75)
\qbezier[70](-17.5,-23.75)(17.5,-23.75)(17.5,-23.75)
\qbezier[70](17.5,-23.75)(27.5,8.5)(27.5,8.5)
\put(0,24.5){\vector(0,1){0}}
\put(-6,2){\vector(3,-1){0}}
\put(6,2){\vector(-3,-1){0}}
\put(-13.5,-17.75){\vector(-2,-3){0}}
\put(4,-6){\vector(-2,3){0}}
\put(21.5,13){\vector(4,-3){0}}
\put(-21.5,13){\vector(-4,-3){0}}
\put(-19.5,-16.75){\vector(1,-3){0}}
\put(13.5,-23.75){\vector(1,0){0}}
\put(19.5,-17.75){\vector(-1,-3){0}}
\put(0,0){\circle*{2}}
\qbezier[60](27.5,8.5)(0,0)(0,0)
\qbezier[60](0,29.5)(0,0)(0,0)
\qbezier[60](-27.5,8.5)(0,0)(0,0)
\qbezier[60](-17.5,-23.75)(0,0)(0,0)
\qbezier[60](17.5,-23.75)(0,0)(0,0)
\put(30.5,8.7){{\tiny\textsf{7}}}
\put(2,31){{\tiny\textsf{5}}}
\put(-30.5,8.7){\llap{{\tiny\textsf{6}}}}
\put(-20.5,-25){\llap{{\tiny\textsf{2}}}}
\put(20.5,-27){{\tiny\textsf{3}}}
\put(2,4){{\tiny\textsf{4}}}
\put(13,21.5){{\tiny I}}
\put(-13,21.5){\llap{{\tiny II}}}
\put(-25.5,-10){\llap{{\tiny III}}}
\put(-3.5,-31){{\tiny IV}}
\put(25.5,-10){{\tiny V}}
\put(10,-3){{\tiny VI}}
\put(1.5,13.5){{\tiny VII}}
\put(-7,5.5){\llap{{\tiny VIII}}}
\put(-8.5,-10){\llap{{\tiny IX}}}
\put(7,-15){\llap{{\tiny X}}}
\put(-17.5,-23.75){\vector(0,-1){9}}
\put(17.5,-23.75){\vector(0,-1){9}}
\put(-16.5,-32.45){{\tiny\textsf{0}}}
\put(16.5,-32.45){\llap{\tiny\textsf{1}}}
\end{picture}
}
}
}
\caption{Several $\Lambda$-\/shaped and $\Pi$-\/shaped terms from the result $\Or(\boldsymbol{\gamma}_5)$ of orienting to Kontsevich orgraphs the pentagon 
wheel graph in the 
cocycle~$\boldsymbol{\gamma}_5$.}\label{FigPenta}
\end{figure}
from the right\/-\/hand side $\cQ_5=\Or(\boldsymbol{\gamma}_5)(\cP)$ of the flow determined by the Kontsevich\/--\/Willwacher pentagon\/-\/wheel cocycle 
$\boldsymbol{\gamma}_5\in\ker\Id$ (see~\cite{OrMorphism,sqs17} and~\cite{JNMP2017}):
\begin{multline*}
\cQ_5=(+2)\cdot\left(
\begin{smallmatrix}
 \mathsf{0} & \mathsf{1} \\
 S_0 & S_1 
\end{smallmatrix}\ %
\begin{smallmatrix}
 \mathsf{2} & \mathsf{4} \\
 VI & I 
\end{smallmatrix}\ %
\begin{smallmatrix}
 \mathsf{2} & \mathsf{5} \\
 VII & II 
\end{smallmatrix}\ %
\begin{smallmatrix}
 \mathsf{2} & \mathsf{6} \\
 VIII & III 
\end{smallmatrix}\ %
\begin{smallmatrix}
 \mathsf{2} & \mathsf{7} \\
 IX & IV 
\end{smallmatrix}\ %
\begin{smallmatrix}
 \mathsf{2} & \mathsf{3} \\
 X & V 
\end{smallmatrix}
\right)
+{}\\
{}+ 10\cdot\left(
\begin{smallmatrix}
 \mathsf{0} & \mathsf{1} \\
 S_0 & S_1  
\end{smallmatrix}\ %
\begin{smallmatrix}
 \mathsf{2} & \mathsf{4} \\
 IV & X 
\end{smallmatrix}\ %
\begin{smallmatrix}
 \mathsf{2} & \mathsf{5} \\
 IX & VI 
\end{smallmatrix}\ %
\begin{smallmatrix}
 \mathsf{3} & \mathsf{6} \\
 V & I 
\end{smallmatrix}\ %
\begin{smallmatrix}
 \mathsf{4} & \mathsf{7} \\
 VII & II 
\end{smallmatrix}\ %
\begin{smallmatrix}
 \mathsf{2} & \mathsf{4} \\
 III & VIII 
\end{smallmatrix}
\right)
+{}\\
{}+
10\cdot\left(
\begin{smallmatrix}
 \mathsf{0} & \mathsf{3} \\
 S_0 & IV 
\end{smallmatrix}\ %
\begin{smallmatrix}
 \mathsf{1} & \mathsf{4} \\
 S_1 & X 
\end{smallmatrix}\ %
\begin{smallmatrix}
 \mathsf{2} & \mathsf{5} \\
 IX & VII 
\end{smallmatrix}\ %
\begin{smallmatrix}
 \mathsf{6} & \mathsf{7} \\
 II & I 
\end{smallmatrix}\ %
\begin{smallmatrix}
 \mathsf{2} & \mathsf{4} \\
 III & VIII 
\end{smallmatrix}\ %
\begin{smallmatrix}
 \mathsf{3} & \mathsf{4} \\
 V & VI 
\end{smallmatrix}
\right)
+\cdots.
\end{multline*}
The first and second graphs, which we denote by~$\{a\}$ and~$\{b\}$, are $\Lambda$-\/shaped whereas the third graph~$\{c\}$ is $\Pi$-\/shaped; there are 167~terms in~$\cQ_5$, of which some are grouped in pairs so that there are 91 bi\/-\/vector terms in total: of them, 15 orgraphs are $\Lambda$-\/shaped and the rest, $\Pi$-\/shaped, undergo the skew\/-\/symmetrisation.

The transition~$\{a\}\longmapsto\{c\}$ employs the following sign matching factors:\footnote{This example of transition between orgraphs, $\{a\}\rightleftarrows\{c\}$ as well as $\{b\}\rightleftarrows\{c\}$,
is instructive also in that the number of inversions, \textsl{i.e.}\ outgoing edge pairs Left~$<$ Right such that Left~$\succ$ Right, does change parity in the course of $\{a\}$, $\{b\}\rightleftarrows\{c\}$ (specifically, from $5$ and~$3$ to~$2$) but does \emph{not} contribute to the signs in front of the orgraph multiplicities.}
\begin{itemize}
\item Rule~\ref{RuleReverseSinksOrder} for~$\{c\}$ having $(S_0\:IV)\:(S_1\:X)\cdots$ contributes with~$(-)$.
\item The number of arrow reversals in the body of orgraph in the course of transition $\{a\}\longmapsto\{c\}$ equals~$4$ (specifically, these are edges $I$, $V$, $VII$, and~$IX$), whence~$(-)^4=(+)$ by Rule~\ref{RuleRevers}.
\item The transition $\Lambda\rightleftarrows\Pi$ between different shapes yields the universal sign factor~$(-)$.
\end{itemize}
In total, we have that for~$\{a\}\longmapsto\{c\}$, the overall sign is~$(-)\cdot(+)\cdot(-)=(+)$.

Counting the parity of three permutations of the edges $S_0\prec S_1\prec I\prec\ldots\prec X$ in the graphs~$\{a\}$, $\{b\}$, $\{c\}$ is left as an exercise,\footnote{The respective numbers of elementary transpositions are~$10$, $24$, and~$26$.}
cf.~\eqref{EqEpsilon}. 
\end{example}

\begin{example}\label{ExPentaMustache2}
The same\/-\/shape $\Lambda\rightleftarrows\Lambda$-\/transition $\{a\}\rightleftarrows\{b\}$ in the pentagon\/-\/wheel flow~$\cQ_5=\Or(\boldsymbol{\gamma}_5)(\cP)$, see previous example, amounts to the reversal of four arrows in the body of orgraph (specifically, the edges $IV^{\{b\}}=\mathsf{2}$-$\mathsf{3}$, $V^{\{b\}}=\mathsf{3}$-$\mathsf{5}$,
$VI^{\{b\}}=\mathsf{4}$-$\mathsf{5}$, and $IX^{\{b\}}=\mathsf{2}$-$\mathsf{4}$).
Rule~\ref{RuleRevers} tells us at once that the orgraph multiplicities,
$2$ for~$\{a\}$ and $10$ for~$\{b\}$, are taken with equal signs (here, ${+2}:{+10}$).
\end{example}

Rules~\ref{RuleReverseSinksOrder} and~\ref{RuleRevers} completely determine the signs of all Kontsevich orgraphs (counted with their multiplicities) as long as they are obtained by orienting a given graph~$\gamma_a$ in a cocycle~$\gamma=\sum_a c_a\cdot\gamma_a
$.

\smallskip
Finally, let $\gamma_a$ and~$\gamma_b$ be topologically nonisomorphic unoriented graphs in a cocycle~$\gamma\in\ker\Id$
such that the differentials~$\Id(\gamma_a)$ and~$\Id(\gamma_b)$ have a least one nonzero unoriented graph in common.

\begin{Rule}
The matching of signs for --\,clearly, topologically nonisomorphic\,-- Kontsevich orgraphs which appear under~\eqref{EqOr} in the course of orienting different terms, $\gamma_a$ and~$\gamma_b$, in a cocycle~$\gamma=\sum_s c_s\cdot\gamma_s\in\ker\Id$ is provided by the cocycle itself, that is, by the coefficients~$c_s$ and respective edge orderings~$\sE(\gamma_a)$ and~$\sE(\gamma_b)$.
\end{Rule}

The signs in front of encodings of all the Kontsevich orgraphs are thus determined for the linear combination $\Or(\gamma)(\cP)$. Now, in each orgraph (and -- independently from other orgraphs), one can swap, at a price of the minus sign factor, the Left and Right outgoing edges issued from any vertex. For example, this is done during the \emph{normalisation} of encodings, when an orgraph, given in terms of $n$~pairs of $n+2$ target vertices, is realised by using a \emph{minimal} base-\/$(n+2)$ positive number.\footnote{%
The normalisation of orgraph encodings, which can be performed independently for different graphs, can actually make it harder to count the number of arrow reverses in the course of transitions which are controlled by Rule~\ref{RuleRevers}.}

\begin{define}
An \textsl{inversion} is a situation where, at a vertex of a Kontsevich orgraph, Left~$\succ$ Right in the overall edge ordering $S_0\prec S_1\prec I\prec II\prec\cdots$.
\end{define} 

\begin{Rule}\label{RuleInvers}
For any Kontsevich orgraph~$\Gamma$ obtained from~$\Gamma_0$ by
relabelling vertices and possibly, for some of the internal vertices, swapping the consecutive order of two edges issued from any such vertex,
we have
\[
\sign(\Gamma)=\frac{ (-)^{\#\text{inversions}\,(\Gamma)} }{ (-)^{\#\text{inversions}\,(\Gamma_0)} }\cdot\sign(\Gamma_0).
\]
Indeed, permutations of vertices induce parity\/-\/even permutations in the ordered string of edges, whereas each elementary transposition --\,within a pair of edges referred to a specific vertex\,-- is parity\/-\/odd.
\end{Rule}

\centerline{\rule{1in}{0.7pt}}

\begin{rem}
Apart from the $\dd_\cP$-\/nontrivial linear scaling $\dot{\cP}=\cP$,
the only $\dd_\cP$-\/(non)\/tri\-vi\-al, nonlinear and proper ($\not\equiv0$)
flows $\dot{\cP}=\cQ(\cP)$ on spaces of Poisson structures which are known so far 
(cf.~\cite{OpenPrb19})
are only those $\cQ=\Or(\gamma)$ which are obtained by orienting $\Id$-\/cocycles, that is, graphs $\gamma\in\ker\Id$ without multiple edges. In consequence, none of the known orgraphs~$\cQ(\cP)$ contains any two\/-\/cycles $\bullet\rightleftarrows\bullet$. All the more surprising 
it is that orgraphs which do contain such two\/-\/cycles are \textsl{dominant} at the order~$\hbar^4$ in the expansion of Kontsevich $\star$-\/product (presumably, so they are at higher orders of the parameter~$\hbar$), see~\cite{cpp} and~\cite{CattaneoFelder2000,KontsevichFormality}.
\end{rem}

It would also be interesting to apply the technique of infinitesimal deformations, $\dot{\cP}=\Or(\gamma)(\cP)$, of Poisson structures~$\cP$ by using graph complex cocycles~$\gamma$ and the orientation morphism~$\Or$, and the technique of formal deformations $\cP\longmapsto\cP[\hbar]$ of Poisson structures by using the noncommutative $\star$-\/product
(see~\cite{Ascona96,KontsevichFormality} and~\cite{CattaneoFelder2000}) to deformations and \emph{deformation} quantisation of minimal surfaces which are specified by the Schild action functional~\cite{MKMinSurf2019}.

\appendix
\section{The proof of Proposition~\protect\ref{LemmaCommutatorGraSym}}\label{AppLieAlgMor}
The Lie bracket of unoriented graphs $\bigl(\gamma_1$,\ $\sE(\gamma_1)\bigr)$ and $\bigl(\gamma_2$,\ $\sE(\gamma_2)\bigr)$ on $n_i$~vertices and $2n_i-2$ edges in each term is, effectively,
\[
\bigl(\gamma_1 \mathbin{\vec{\circ}} \gamma_2 - \gamma_1 \mathbin{\cev{\circ}} \gamma_2,
\quad \sE(\gamma_1)\wedge\sE(\gamma_2) \bigr).
\]
The commutator $\bigl[\vec{d}/d\veps_1, \vec{d}/d\veps_2\bigr](\cP)$ of the flows
$\tfrac{d}{d\veps_i}(\cP)=\cQ_i(\cP)=\Or(\gamma_i)\bigl(\cP^{\otimes^{n_i}}\bigr)$ 
amo\-unts to the consecutive insertions of the bi\/-\/vector $\cQ_i(\cP)$ instead of a copy of the bi\/-\/vector~$\cP$ in one vertex of the orgraph~$\cQ_{2-i}$. The claim is that
\begin{equation}\label{EqLieAlgMor}
\bigl[ \Or(\gamma_1)\bigl(\cP^{\otimes^{n_1}}\bigr), \Or(\gamma_2)\bigl(\cP^{\otimes^{n_2}}\bigr) \bigr]_{\text{flows}} = \Or\bigl([\gamma_1,\gamma_2]_{\text{graphs}}\bigr)\bigl(\cP^{\otimes^{n_1 + n_2 - 1}}\bigr).
\end{equation}
Let us show that the minuend in the left\/-\/hand side is equal to the minuend in the right\/-\/hand side, and the same for the subtrahends.

\textbf{The left\/-\/hand side.}\quad The two orgraphs $\gamma_i$ are oriented independently from each other. The resulting orgraphs are built of wedges (such that the body edges get oriented in all admissible ways). Every such Kontsevich orgraph either is automatically skew\/-\/symmetric w.r.t.\ the content $f$,\ $g$ of sinks (\textsl{i.e.} with respect to the ordered pair of arguments of this bi\/-\/vector) or it is skew\/-\/symmetrised by the mechanism which already worked in Corollary~\ref{CorSkewGuaranteed}. Namely, the \emph{difference} of orgraphs with the identical labelling of ordered edges, $S_0\prec S_1\prec\sE(\gamma)$ or $\sE(\gamma)\prec S_0\prec S_1$, 
\begin{figure}[htb]
{\unitlength=0.6mm
\centerline{%
\begin{picture}(40,30)
\put(20,24){\oval(40,12)}
\put(10,24){\circle*{2}}
\put(30,24){\circle*{2}}
\put(10,24){\vector(0,-1){20}}
\put(30,24){\vector(0,-1){20}}
\put(-1,25){\llap{$\gamma$}}
\put(9,7){\llap{$S_0$}}
\put(31,7){$S_1$}
\put(9,0){$f$}
\put(29,0){$g$}
\end{picture}
\raisebox{13.5mm}[0pt][0pt]{\quad$-$\quad}
\begin{picture}(40,30)
\put(20,24){\oval(40,12)}
\put(10,24){\circle*{2}}
\put(30,24){\circle*{2}}
\put(10,24){\vector(0,-1){20}}
\put(30,24){\vector(0,-1){20}}
\put(0,18){\llap{$\gamma$}}
\put(9,7){\llap{$S_0$}}
\put(31,7){$S_1$}
\put(9,0){$g$}
\put(29,0){$f$}
\end{picture}
\raisebox{13.5mm}[0pt][0pt]{\quad$=$\quad}
\begin{picture}(40,30)
\put(20,24){\oval(40,12)}
\put(10,24){\circle*{2}}
\put(30,24){\circle*{2}}
\put(10,24){\vector(0,-1){20}}
\put(30,24){\vector(0,-1){20}}
\put(0,18){\llap{$\gamma$}}
\put(9,7){\llap{$S_0$}}
\put(31,7){$S_1$}
\put(9,0){$f$}
\put(29,0){$g$}
\end{picture}
\raisebox{13.5mm}[0pt][0pt]{\quad$+$\quad}
\begin{picture}(40,30)
\put(20,24){\oval(40,12)}
\put(10,24){\circle*{2}}
\put(30,24){\circle*{2}}
\put(10,24){\vector(1,-1){20}}
\put(30,24){\vector(-1,-1){20}}
\put(0,18){\llap{$\gamma$}}
\put(9,7){\llap{$S_0$}}
\put(31,7){$S_1$}
\put(9,0){$f$}
\put(29,0){$g$}
\end{picture}
}
}
\end{figure}
but with the content of the two sinks swapped
is equal to the \emph{sum} of orgraphs with the identical labelling of the body edges but with the tails of the arrows~$S_0$ (heading to~$f$) and~$S_1$ (heading to~$g$) swapped.

\textbf{The right\/-\/hand side: minuend.}\quad Summing over vertices and attachments, we replace a vertex~$v_0$ in the ``victim'' graph~$\gamma_2$ by the graph~$\gamma_1$ on $n_1$~vertices and $2n_1-2$ edges. In the graph~$\gamma_2$, the edges which were incident to the blown\/-\/up vertex~$v_0$ are now attached --\,in all possible ways\,-- to some vertices of the inserted graph~$\gamma_1$.
Note that if two such edges, $vu$ and $v' u$ now connect two distinct vertices, $v$ and~$v'$, of the victim graph~$\gamma_2$ with the same vertex~$u$ of the graph~$\gamma_1$, then one of the two edges precedes the other with respect to the old edge ordering in the victim graph.
Likewise, if two such edges, $vv_0$ and~$v'v_0$ (for which the ordering was defined), now connect by~$vu$ and~$v'u'$ two distinct vertices, $v$ and~$v'$, in the victim graph~$\gamma_2$ with two distinct vertices, $u$ and~$u'$ in the graph~$\gamma_1$, then the insertion $\gamma_1\mathbin{\vec{\circ}}\gamma_2$ contains another graph in which the only difference from the above is that the two edges~$vv_0$ and~$v'v_0$ become~$v u'$ and~$v' u$ (but all the other edges~$v_0w$ in the body of the victim graph~$\gamma_2$ are attached to vertices of the graph~$\gamma_1$ in the same way as they are in the former case).

In every term of the graph $\gamma_1\mathbin{\vec{\circ}}\gamma_2$, consider the subgraph~$\gamma_1$; it remains intact in the course of insertion~$\vec{\circ}$. When the big graph $\gamma_1\mathbin{\vec{\circ}}\gamma_2$ is oriented by $\Or(\cdot)(\cP)$, so is the subgraph~$\gamma_1$. There were $2n_1-2$ edges in the (body of the) graph~$\gamma_1$; none of these edges, still between two vertices of the (sub)\/graph~$\gamma_1$, can be oriented using any wedge issued from a vertex of the outer graph~$\gamma_2$. This implies that exactly $2n_1-2$ arrows belonging to the $n_1$~bi\/-\/vector wedges are spent on orienting the body of the subgraph~$\gamma_1$ in the big graph $\gamma_1\mathbin{\vec{\circ}}\gamma_2$. Only two arrows leave the subgraph~$\gamma_1$: they head either to one or two sinks of the orgraph $\Or(\gamma_1\mathbin{\vec{\circ}}\gamma_2)$ or to a vertex\footnote{%
It cannot be that two arrows, $\overrightarrow{uv}$ and $\overrightarrow{u'v}$, from the subgraph~$\gamma_1$ head towards the same vertex~$v$ in the victim graph~$\gamma_2$ because, with regards to the old topology of~$\gamma_2$ in which a vertex~$v_0$ will be replaced by the graph~$\gamma_1$, this would mean a double edge~$v_0v$, hence $\gamma_2$~was a zero graph.}
or two vertices in the rest of the victim graph~$\gamma_2$, \textsl{i.e.} excluding the blown\/-\/up vertex~$v_0$. All the other edges which were of the form~$v_0v$ in the graph~$\gamma_2$ now become arrows $\overrightarrow{vu}$ heading towards vertices~$u$ of the subgraph~$\gamma_1$ in the big graph $\gamma_1\mathbin{\vec{\circ}}\gamma_2$.

To establish the equality of the minuend in the left\/-\/hand side of~\eqref{EqLieAlgMor} to the minuend in the right\/-\/hand side of that formula, it remains to recall that by construction, all body edges of the graph~$\gamma_2$ antecede those of~$\gamma_1$ (and \textit{vice versa}: body edges of the graph~$\gamma_1$ precede those of the graph~$\gamma_2$), so that now, the ordering $S_0^{(1)}\prec S_1^{(1)}$ of the arrows which are issued to the arguments of the bi\/-\/vector $\cQ_1(\cP)=\Or(\gamma_1)(\cP)$ is always dictated by the ordering $\sE(\gamma_2)\wedge S_0^{(2)} \wedge S_1^{(2)}$ of two edges from the (or)\/graph~$\Or(\gamma_2)$.

The subtrahends in which the graph~$\gamma_2$ is inserted into some vertex 
of the graph~$\gamma_1$ are processed in an analogous way.
The proof is complete.

\section{The tetrahedron: its $\Pi$-shaped orientation skew\/-\/symmetrized}\label{AppRotate}
The edge orderings $\sE(c)=\bigl(S_0\wedge S_1\wedge I\wedge\ldots\wedge VI\bigr)^{(c)}$ and 
$\sE(b)=\bigl(S_0\wedge S_1\wedge I\wedge\ldots\wedge VI\bigr)^{(b)}$ are
\begin{figure}[htb]
\begin{center}
{\unitlength=1mm
\begin{picture}(30,45)(-15,-19)
\put(-16,18){$(c)$}
\qbezier(-8,22)(12,30)(6,22)
\put(6,22){\vector(-2,-3){0}}
\put(7,20){{\small $120^\circ$}}
\put(-15,0){\vector(1,0){30}}
\put(-15,0){\vector(3,4){15}}
\put(0,20){\vector(0,-1){29}}
\put(0,20){\vector(3,-4){15}}
\put(15,0){\vector(-3,-2){15}}
\put(0,-10){\vector(-3,2){14}}
\put(0,-10){\vector(0,-1){10}}
\put(15,0){\vector(0,-1){10}}
\put(-16,-1){\llap{{\footnotesize 3}}}
\put(16,-1){{{\footnotesize 4}}}
\put(-0.75,-13){\llap{{\footnotesize 2}}}
\put(-0.75,-21){\llap{{\footnotesize 0}}}
\put(14.25,-11){\llap{{\footnotesize 1}}}
\put(-0.5,20.5){{{\footnotesize 5}}}
\put(-9,-7.5){\llap{\small I}}
\put(9,-7.5){\small II}
\put(0.75,6){\small III}
\put(-8,0.75){\small IV}
\put(-10,7){\llap{\small V}}
\put(10,7){\small VI}
\put(0.5,-17){{\small $S_0$}}
\put(16,-7.5){\small $S_1$}
\end{picture}
\qquad\raisebox{18mm}{$\longmapsto$}\qquad
\begin{picture}(30,45)(-15,-19)
\put(-16,18){$(c)$}
\put(-15,0){\vector(1,0){30}}
\put(0,20){\vector(-3,-4){15}}
\put(0,20){\vector(0,-1){29}}
\put(15,0){\vector(-3,4){15}}
\put(15,0){\vector(-3,-2){15}}
\put(0,-10){\vector(-3,2){14}}
\put(0,-10){\vector(0,-1){10}}
\put(-15,0){\vector(0,-1){10}}
\put(-16,-1){\llap{{\footnotesize 2}}}
\put(16,-1){{{\footnotesize 3}}}
\put(-0.75,-13){\llap{{\footnotesize 4}}}
\put(-0.75,-21){\llap{{\footnotesize 1}}}
\put(-14.25,-11){{{\footnotesize 0}}}
\put(-0.5,20.5){{{\footnotesize 5}}}
\put(-9,-7.5){\llap{\small II}}
\put(9,-7.5){\small IV}
\put(0.75,6){\small VI}
\put(-8,0.75){\small I}
\put(-10,7){\llap{\small III}}
\put(10,7){\small V}
\put(0.5,-17){{\small $S_1$}}
\put(-16,-7.5){\llap{\small $S_0$}}
\end{picture}
\qquad\raisebox{18mm}{$\simeq$}\qquad
\begin{picture}(30,45)(-15,-19)
\put(-16,18){$(b)$}
\put(-15,0){\vector(1,0){30}}
\put(0,20){\vector(-3,-4){15}}
\put(0,20){\vector(0,-1){29}}
\put(15,0){\vector(-3,4){15}}
\put(15,0){\vector(-3,-2){15}}
\put(0,-10){\vector(-3,2){14}}
\put(0,-10){\vector(0,-1){10}}
\put(-15,0){\vector(0,-1){10}}
\put(-0.75,-21){\llap{{\footnotesize 0}}}
\put(-14.25,-11){{{\footnotesize 1}}}
\put(-9,-7.5){\llap{\small I}}
\put(9,-7.5){\small II}
\put(0.75,6){\small III}
\put(-8,0.75){\small IV}
\put(-10,7){\llap{\small VI}}
\put(10,7){\small V}
\put(0.5,-17){{\small $S_0$}}
\put(-16,-7.5){\llap{\small $S_1$}}
\end{picture}
}
\end{center}
\end{figure}
related by the equalities $S_0^{(c)}=S_1^{(b)}$, $S_1^{(c)}=S_0^{(b)}$,
$I^{(c)}=IV^{(b)}$, $II^{(c)}=I^{(b)}$, $III^{(c)}=V^{(b)}$, $IV^{(c)}=II^{(b)}$, 
$V^{(c)}=VI^{(b)}$, and $VI^{(c)}=III^{(b)}$, whence one easily verifies that
\[
\sE(c) = - (-)^6\, \sE(b),
\]
the leading minus coming from the relabelling $S_0\rightleftarrows S_1$ and the rest from the permutation of body edges.
We conclude that the arithmetic sum of two Kontsevich orgraphs $(b$--$c)$ in Fig.~\ref{FigTetra} on p.~\pageref{FigTetra} is the skew\/-\/symmetrisation of the $\Pi$-\/shaped orientation of the tetrahedron~$\boldsymbol{\gamma}_3$ by using arrow wedges.

\section*{Acknowledgements}
The authors thank the Organisers of international workshop `Homotopy algebras, deformation theory and quantization' (16--22~September 2018 in Bed\-le\-wo, Poland) for helpful discussions and warm atmosphere during the meeting.
The authors are grateful to the anonymous referee for remarks and suggestions, and to
G.~Felder, S.~Gutt, and M.~Kontsevich for helpful discussion. 
A part of this research was done while RB was visiting at~RUG and AVK was visiting at
the $\smash{\text{IH\'ES}}$ in Bures\/-\/sur\/-\/Yvette, France and at the
JGU~Mainz (supported by~IM JGU via project 5020 and JBI~RUG project~106552).
The research of AVK~was supported by the $\smash{\text{IH\'ES}}$ (partially, by the Nokia Fund). 


\begin{thebibliography}{77}

\bibitem{MKMinSurf2019}
\by{J. Arnlind, 
J. Hoppe, 
M. Kontsevich}, 
\textit{Quantum minimal surfaces}, 
\textrm{arXiv:1903.10792} [math-ph]

\bibitem{f16}
\by{A. Bouisaghouane, R. Buring, A. Kiselev},
\textit{The Kontsevich tetrahedral flow revisited}, \jour{J.~Geom.\ Phys.} 
\vol{119} 
(2017), 272--285.\ 
(\jour{Preprint} \textrm{arXiv:1608.01710} 
[q-alg])

\bibitem{Brown2012}
\by{F. Brown}, \textit{Mixed Tate motives over $\mathbb{Z}$}, \jour{Annals of Math.} \vol{175}
(2012), 949--976.

\bibitem{sqs15}
\by{R. Buring, A. V. Kiselev}, \textit{On the Kon\-tse\-vich $\star$-\/product asso\-ci\-a\-ti\-vi\-ty me\-cha\-n\-ism}, \jour{PEPAN Letters} \vol{14}:2 \textit{Supersymmetry and Quantum Symmetries'2015}
(2017), 403--407.\ %
(\jour{Preprint} \textrm{arXiv:1602.09036} [q-alg])

\bibitem{OrMorphism}
\by{R. Buring,  A. V. Kiselev}, \textit{The orientation morphism: from graph cocycles to deformations of Poisson structures}, in:
\jour{J.~Phys.\textup{:}\ Conf.\ Ser.} \vol{1194} (2019),
Proc.\ 32nd~Int.\ colloquium on Group\/-\/theoretical methods in Physics: \textsc{Group32} (9--13~July 2018, CVUT Prague, Czech Republic), Paper~012017, 1--10.\ 
(\jour{Preprint} \textrm{arXiv:1811.07878} [math.CO])

\bibitem{cpp}
\by{R. Buring,  A. V. Kiselev},
\textit{The expansion $\star$ \textup{mod}~$\bar{o}(\hbar^4)$
and computer\/-\/assisted proof schemes in the Kon\-tse\-vich deformation quantization},
\jour{Experimental Math.} (2019), 54~p.,
in press, doi:10.1080/10586458.2019.1680463.
(\jour{Preprint} $\smash{\text{IH\'ES}}$/M/17/05, 
\textrm{arXiv:1702.00681}~
[math.CO]) 

\bibitem{Kiev18}
\by{R. Buring,  A. V. Kiselev}, 
\textit{Formality morphism as the mechanism of $\star$-\/product associativity: how it works},
\jour{Collection of works 
Inst.\ Math., Kyiv} 
\vol{16}:1 (2019), \textit{Symmetry \textsl{\&} Integrability of Equations of Mathematical Physics},
22--43.\ (\jour{Preprint} \textrm{arXiv:1907.00639} [q-alg])

\bibitem{sqs19}
\by{R. Buring,  A. V. Kiselev}, \textit{Universal cocycles and the graph complex action on homogeneous Poisson brackets by diffeomorphisms},
\jour{PEPAN Letters} (in press) \textit{Supersymmetry and Quantum Symmetries'2019} (2020), 8~pages.\ %
(\jour{Preprint} $\smash{\text{IH\'ES}}$/M/19/20 (2019),
\textrm{arXiv:1912.12664} [math.SG]) 

\bibitem{JNMP2017}
\by{R. Buring, A. V. Kiselev, N. J. Rutten},
\textit{The heptagon\/-\/wheel cocycle in the Kon\-tse\-vich graph complex},
\jour{J.~Nonlin.\ Math.\ Phys.} 
\vol{24} (2017), Suppl.~1 
`Local \& Nonlocal Sym\-me\-t\-ri\-es in Mathematical Physics', 
157--173.\ (\jour{Preprint} \textrm{arXiv:1710.00658}~[math.CO])

\bibitem{sqs17}
\by{R. Buring,  A. V. Kiselev,  N. J. Rutten},
\textit{Poisson brackets symmetry from the pentagon\/-\/wheel cocycle in the graph complex},
\jour{Phy\-sics of 
Particles and 
Nuclei} 
\vol{49}:5 \textit{Supersymmetry and Quantum Symmetries'2017} (2018),
924--928.\ 
(\jour{Preprint} \textrm{arXiv:1712.05259} \mbox{[math-ph]})

\bibitem{CattaneoFelder2000}
\by{A. S. Cattaneo, G. Felder},
\textit{A~path integral approach to the Kontsevich quantization formula},
\jour{Comm.\ Math.\ Phys.} \vol{212}:3 (2000),
591--611.\ (\jour{Preprint} \textrm{arXiv:q-alg/9902090})

\bibitem{Jost2013}
\by{C. Jost},
\textit{Globalizing $L_\infty$-\/automorphisms of the Schouten algebra of polyvector fields}, 
\jour{Differential Geom.\ Appl.} \vol{31}:2 (2013), 239--247.
(\jour{Preprint} \textrm{arXiv:1201.1392} [q-alg])

\bibitem{OpenPrb19}
\by{A. V. Kiselev}, \textit{Open problems in the Kontsevich graph construction of Poisson bracket symmetries}, in:
\jour{J.~Phys.\textup{:}\ Conf.\ Ser.} \vol{1416} (2019), 
Proc.\ XXVI Int.\ conf. `Integrable Systems \& Quantum Symmetries' 
(8--12~July 2019, CVUT Prague, Czech Republic), Paper 012018, 1--8.
(\jour{Preprint} \textrm{arXiv:1910.05844} [math-ph])

\bibitem{MKParisECM}
\by{M. Kontsevich},
\textit{Feynman diagrams and low\/-\/dimensional topology}, in:
\book{First Europ.\ Congr.\ of Math.} \vol{2} (Paris, 1992), 
Progr.\ Math. \vol{120}, Birkh\"auser, Basel, 1994, 97--121. 

\bibitem{MKZurichICM}
\by{M. Kontsevich},
\textit{Homological algebra of mirror symmetry}, in:
\book{Proc.\ Intern. Congr. Math.}~\vol{1} 
(Z\"urich, 1994), 
Birkh\"auser, Basel, 1995, 120--139.

\bibitem{Ascona96}
\by{M. Kontsevich},
\textit{Formality conjecture}, in:
\book{Deformation theory and symplectic geometry} (Ascona 
1996), D.\,
Sternheimer, J.\,
Rawnsley and S.\,
Gutt (eds.), 
Math.\ Phys.\ Stud.~\vol{20}, Kluwer Acad.\ Publ., Dordrecht, 1997, 139--156.

\bibitem{Operads1999}
\by{M. Kontsevich},
\textit{Operads and motives in deformation quantization},
\jour{Lett.\ Math.\ Phys.} \vol{48}:1 (1999), \textit{Mosh\'e Flato (1937--1998)}, 35--72. (\textit{Preprint} \textrm{arXiv:math.QA/9904055}
)

\bibitem{KontsevichFormality}
\by{M. Kontsevich},
\textit{Deformation quantization of {P}oisson manifolds},
\jour{Lett.\ Math.\ Phys.}~\vol{66}:3 (2003), 157--216.
(\textit{Preprint} \textrm{arXiv:q-alg/9709040})

\bibitem{Bourbaki2017}
\by{M. Kontsevich},
\textit{Derived Grothendieck\/--\/Teichm\"uller group and graph complexes [after T.~Will\-wa\-cher]}, in: \jour{S\'eminaire Bourbaki} (69\`eme ann\'ee, 2016--2017), Janvier 2017,\ No.~1126 (2017), 183--212.

\bibitem{Identities18}
\by{N. J. Rutten, A. V. Kiselev},
\textit{The defining properties of the Kontsevich unoriented graph complex}, in:
\jour{J.~Phys.\textup{:}\ Conf.\ Ser.} \vol{1194} (2019), 
Proc.\ 32nd~Int.\ colloquium on Group\/-\/theoretical methods in Physics: \textsc{Group32} (9--13~July 2018, CVUT Prague, Czech Republic), Paper~012095, 1--10.\ 
(\jour{Preprint} \textrm{arXiv:1811.10638} [math.CO])

\bibitem{WillwacherGRT}
\by{T. Willwacher}, 
\textit{M.~Kontsevich's graph complex and the Grothendieck\/--\/Teichm\"uller Lie algebra},
\jour{Invent.\ Math.} \vol{200}:3 (2015), 671--760.\ %
(\jour{Preprint} \textrm{arXiv:1009.1654} [q-alg]) 

\bibitem{WillwacherZivkovic2015Table}
\by{T. Willwacher, M. \v{Z}ivkovi\'{c}},
\textit{Multiple edges in M.~Kontsevich's graph complexes and computations of the dimensions and Euler characteristics}, \jour{Adv.\ Math.} \vol{272} (2015), 553--578.\ %
(\jour{Preprint} \textrm{arXiv:1401.4974} 
[q-alg])


\end{thebibliography}
\end{document}